\documentclass[a4paper]{amsart}
\usepackage{amsmath,amsthm,amssymb,latexsym,epic,bbm,comment, mathrsfs}
\usepackage{graphicx,enumerate,stmaryrd,color}
\usepackage[all,2cell]{xy}
\xyoption{2cell}
\usepackage[active]{srcltx}
\usepackage[parfill]{parskip}
\usepackage{enumerate}
\usepackage{tikz-cd}

\newtheorem{theorem}{Theorem}
\newtheorem{lemma}[theorem]{Lemma}
\newtheorem{cor}[theorem]{Corollary}
\newtheorem{prop}[theorem]{Proposition}
\newtheorem{conjecture}[theorem]{Conjecture}

\theoremstyle{definition}
\newtheorem{definition}[theorem]{Definition}
\newtheorem{rmk}[theorem]{Remark}

\DeclareMathOperator{\im}{\mathrm{im}}

\DeclareMathOperator{\Hom}{\mathrm{Hom}}
\DeclareMathOperator{\End}{\mathrm{End}}

\begin{document}
\pagestyle{plain}
\title{On simple transitive 2-representations\\ of bimodules over the dual numbers}
\author{Helena Jonsson}

\begin{abstract}
We study the problem of classification of simple transitive 2-representations
for the (non-finitary) $2$-category of bimodules over the dual numbers.
We show that simple transitive 2-representations with finitary apex 
are necessarily of rank $1$ or $2$, and those of rank $2$ are exactly
the cell $2$-representations. For $2$-representations of rank $1$,
we show that they cannot be constructed using the approach
of (co)algebra $1$-morphisms. We also propose an alternative 
definition of (co-)Duflo $1$-morphisms for finitary
$2$-categories and describe them
in the case of bimodules over the dual numbers.
\end{abstract}

\maketitle

\section{Motivation, introduction and description of the results}

Classification problems form an important and intensively studied class of 
questions in modern representation theory. One of the natural examples of 
these kinds of problems is the problem of classification of all ``simple''
representations of a given mathematical object. During the last 20 years,
a lot of attention was attracted to the study of representations of 
tensor categories and $2$-categories, see \cite{EGNO,Ma} and references therein.
In particular, there are by now a number of interesting tensor categories
(and $2$-categories) for which the structure of ``simple'' representations
is well-understood. To the best of our knowledge, the first deep results
of this kind can be found in \cite{Os1,Os2}, we refer to \cite{EGNO}
for more details.

Around 2010, Mazorchuk och Miemietz started a  systematic study of 
representation theory of finitary $2$-categories, see the original
series \cite{MM1,MM2,MM3,MM4,MM5,MM6} of papers by these authors. 
Finitary $2$-categories can be considered as natural $2$-analogues of 
finite dimensional algebras, in particular, they have various 
finiteness properties, analogous to those of the category of 
projective modules over a finite dimensional algebra. The paper
\cite{MM5} introduces  the notion of {\em simple transitive}  
$2$-representations of finitary $2$-categories and provides convincing
arguments, including an adaptation of the Jordan-H{\"o}lder theorem,
on why these $2$-representations are a natural $2$-analogue for the 
notion of a simple module over an associative algebra. This motivated
the natural problem of classification of simple transitive 
$2$-representations for various classes of finitary $2$-categories.
This problem was considered and solved in a number of special cases, 
see \cite{MM5,Zh2,Zi1,Zi2,MZ1,MZ2,MMZ1,MMZ2,MaMa,KMMZ,MT,MMMT,MMMZ,MMMTZ} 
and also \cite{Ma2} for a slightly outdated overview 
on the status of that problem.

Arguably, one of the most natural examples of a finitary $2$-category is the
$2$-category $\mathscr{C}_A$ of {\em projective bimodules} over a finite-dimensional
associative algebra $A$, introduced in \cite[Subsection~7.3]{MM1}.
Classification of simple transitive $2$-representation of
$\mathscr{C}_A$ is given in \cite{MMZ2}, with the special case of a 
self-injective $A$ treated already in \cite{MM5,MM6}. The reason to restrict
to projective bimodules is the observation that, in the general case, 
the tensor category $A\text{-}\mathrm{mod}\text{-}A$ of {\em all}
finite dimensional $A$-$A$-bimodules is not finitary because it has
infinitely many indecomposable objects. The only basic connected
algebras $A$, for which  $A\text{-}\mathrm{mod}\text{-}A$ is finitary,
are the radical square zero quotients of the path algebras of 
uniformly oriented type
$A$ Dynkin quivers, see \cite{MZ2}. Moreover, for almost all $A$,
the category $A\text{-}\mathrm{mod}\text{-}A$ is wild,
that is the associative algebra $A\otimes_{\Bbbk}A^{\mathrm{op}}$,
whose module category is equivalent to $A\text{-}\mathrm{mod}\text{-}A$,
has wild representation type, and hence indecomposable objects
of $A\text{-}\mathrm{mod}\text{-}A$ are not even known
(and, perhaps, never will be known).

The smallest example of the algebra $A$ for which the category 
$A\text{-}\mathrm{mod}\text{-}A$ is not finitary, but is, at least,
tame, is the algebra $D:=\Bbbk[x]/(x^2)$ of {\em dual numbers}.
The combinatorics of tensor product of indecomposable objects
in $D\text{-}\mathrm{mod}\text{-}D$ is described in \cite{Jo1,Jo2}.
In particular, although not being finitary itself, 
$D\text{-}\mathrm{mod}\text{-}D$ has a lot of finitary 
subcategories and subquotients. The main motivation for the present paper
is to understand simple transitive $2$-representations of
$D\text{-}\mathrm{mod}\text{-}D$ which correspond to simple 
transitive $2$-representations of its finitary subquotients.

Our main result is Theorem~\ref{mainthm} which can be found
in Subsection~\ref{s-maint}. It asserts that simple transitive
$2$-representations of $D\text{-}\mathrm{mod}\text{-}D$ with
finitary apex are necessarily of rank $1$ or $2$ and, in the
latter case, each such $2$-representation is necessarily 
equivalent to a so-called {\em cell $2$-representation},
which is an especially nice class of $2$-representations.
Unfortunately, at this stage we are not able  to classify
(or, for that matter, even to construct, with one exception) 
rank $1$ simple transitive $2$-representations. One possible
reason for that is given in Theorem~\ref{nonconstr}
in Subsection~\ref{s-nonc} which asserts that potential 
simple transitive $2$-representations of rank
$1$ cannot be constructed using the approach
of (co)algebra $1$-morphisms, developed in \cite{MMMT}
for the so-called {\em fiat $2$-categories}, that is
finitary $2$-categories with a weak involution and adjunction morphism.
Needless to say, neither $D\text{-}\mathrm{mod}\text{-}D$
nor any of its finitary subquotients is fiat. 

Section~\ref{sdufloo} and \ref{scoalgg} summarize, in some sense,
the outcome of our failed attempt to adjust the approach of
\cite{MMMT} at least for construction of simple transitive
$2$-representations of $D\text{-}\mathrm{mod}\text{-}D$.
Due to the fact that $D\text{-}\mathrm{mod}\text{-}D$
is not fiat, several classical notions for fiat 
$2$-categories require non-trivial adaptation to the
more general setup of $D\text{-}\mathrm{mod}\text{-}D$.
One of these, discussed in detail in Section~\ref{sdufloo},
is that of a {\em Duflo $1$-morphism}. Originally, it is
defined in \cite{MM1} in the fiat setup and slightly
adjusted in \cite{Zh} to a more general finitary setting.
Here we propose yet another alternative definition
of Duflo $1$-morphisms (and the dual notion of co-Duflo
$1$-morphisms) using certain universal properties, 
see Subsections~\ref{s1-duf} and \ref{s2-duf}.
We show in Proposition~\ref{pduf-fin} that our notion agrees
with the notion of Duflo $1$-morphisms from \cite{MM1}
in the fiat case. We show that some left cells
in $D\text{-}\mathrm{mod}\text{-}D$ have a Duflo
$1$-morphism and that some other left cells have
a co-Duflo $1$-morphism, see Subsections~\ref{subs_duflo}
and \ref{s2-duf}. In Section~\ref{scoalgg}, we further
show that these Duflo and co-Duflo $1$-morphisms
admit the natural structure of 
coalgebra and algebra $1$-morphisms, respectively.

All necessary preliminaries are collected in Section~\ref{prelim}.
Our main  Theorem~\ref{mainthm} has four statements.
The first one is proved in Subsection~\ref{s0p}.
The other three are proved in Sections~\ref{s1p}, 
\ref{sec_rank2_cell} and \ref{s3p}, respectively.
\vspace{5mm}

{\bf Acknowledgments.} This research is partially supported by G{\"o}ran Gustafssons Stiftelse. The author wants to thank her supervisor Volodymyr Mazorchuk for many helpful comments and discussions.

\section{Preliminaries}\label{prelim}

\subsection{2-categories}
A \emph{2-category} $\mathscr{C}$ consists of
\begin{itemize}
\item objects $\mathtt{i},\mathtt{j},\mathtt{k},$...;
\item for each pair of objects $\mathtt{i},\mathtt{j}$, a small category 
$\mathscr{C}(\mathtt{i},\mathtt{j})$ of {\em morphisms}
from $\mathtt{i}$ to $\mathtt{j}$,  objects of 
$\mathscr{C}(\mathtt{i},\mathtt{j})$ are called 1-morphisms $F,G,H,$...,
and morphisms of $\mathscr{C}(\mathtt{i},\mathtt{j})$ 
are called 2-morphisms $\alpha ,\beta,$...;
\item for each object $\mathtt{i}$, an identity 1-morphism $\mathbbm{1}_\mathtt{i}$;
\item bifunctorial composition 
$\circ:\mathscr{C}(\mathtt{j},\mathtt{k})\times \mathscr{C}(\mathtt{i},\mathtt{j})
\to \mathscr{C}(\mathtt{i},\mathtt{k})$.
\end{itemize}
This datum is supposed to satisfy the obvious set of strict axioms. 
The internal composition of $2$-morphisms in $\mathscr{C}(\mathtt{i},\mathtt{j})$ 
is called {\em vertical} and denoted by $\circ_v$. The composition of 
$2$-morphisms induced by $\circ$ is called {\em horizontal} and denoted
$\circ_h$.

Let $\Bbbk$ be a field.
Important examples of 2-categories are
\begin{itemize}
\item $\mathbf{Cat}$, the 2-category whose objects are small categories, 1-morphisms are functors, and 2-morphisms are natural transformations of functors;
\item $\mathfrak{A}^f_\Bbbk$, the 2-category whose objects are finitary $\Bbbk$-linear categories, 1-mor\-phisms are additive $\Bbbk$-linear functors, and 2-morphisms are natural transformations of functors;
\item $\mathfrak{R}_\Bbbk$, the 2-category of finitary $\Bbbk$-linear abelian categories, whose objects are small categories equivalent to module categories of finite-dimensional associative $\Bbbk$-algebras, 1-morphisms are right exact additive $\Bbbk$-linear functors, and 2-morphisms are natural transformations of functors.
\end{itemize}

\subsection{2-representations}
Let $\mathscr{C}$ be a 2-category. A \emph{2-representation} of $\mathscr{C}$ is a  strict 2-functor $\mathbf{M}:\mathscr{C}\to \mathbf{Cat}$.

For example, given an object $\mathtt{i}$ in $\mathscr{C}$, we can 
define the \emph{principal representation} $\mathbf{P}_\mathtt{i}=\mathscr{C}(\mathtt{i},-)$. 
A \emph{finitary 2-representation} of $\mathscr{C}$ is a  strict 2-functor $\mathbf{M}:\mathscr{C}\to \mathfrak{A}^f_\Bbbk$.

A finitary 2-representation $\mathbf{M}$ of $\mathscr{C}$ is called \emph{transitive} if, for any indecomposable object $X\in \mathbf{M}(\mathtt{i})$ and $Y\in \mathbf{M}(\mathtt{j})$, there is a 1-morphism $U$ in $\mathscr{D}$ such that $Y$ is isomorphic to a direct summand of $\mathbf{M}(U)X$. 
We, further, say that
$\mathbf{M}$ \emph{simple} if it has no proper nonzero $\mathscr{C}$-stable ideals. While simplicity implies transitivity, 
we follow
\cite{MM5} and speak of simple transitive 2-representations to emphasize the 
two levels (objects and morphisms) of the involved structure.

All 2-representations of $\mathscr{C}$ form a 2-category, 
see \cite[Subsection~2.3]{MM3} for details. In particular,
two 2-representations $\mathbf{M}$ and $\mathbf{N}$ of $\mathscr{C}$ are \emph{equivalent} if there is a 2-natural transformation $\Phi: \mathbf{M}\to \mathbf{N}$ which restricts to an equivalence $\mathbf{M}(\mathtt{i})\to \mathbf{N}(\mathtt{i})$ for every object $\mathtt{i}\in \mathscr{C}$.

If $\mathscr{C}$ has only one object $\mathtt{i}$, we say that a finitary 2-representation $\mathbf{M}$ of $\mathscr{C}$ has \emph{rank} $r$ if the category $\mathbf{M}(\mathtt{i})$ has exactly $r$ isomorphism classes of  indecomposable objects.

\subsection{Abelianization}

For every finitary $2$-representation $\mathbf{M}$ of $\mathscr{C}$,
we can consider its {\em (projective) abelianization}
$\overline{\mathbf{M}}$ as defined in \cite[Section~3]{MMMT}.
Then $\overline{\mathbf{M}}$ is a $2$-functor from $\mathscr{C}$
to $\mathfrak{R}_\Bbbk$ and, up to equivalence,
$\mathbf{M}$ is recovered by restricting $\overline{\mathbf{M}}$
to the subcategories of projective objects in the 
underlying abelian categories of the 
{\em abelian} $2$-representation $\overline{\mathbf{M}}$.

There is also the dual notion of {\em (injective) abelianization}
$\underline{\mathbf{M}}$.

\subsection{Cells and cell 2-representations}
One the set of isomorphism classes of indecomposable 1-morphisms in $\mathscr{C}$, define the \emph{left preorder} $\leq_L$ by $F\leq_L G$ if there is some $H$ such that $G$ is a direct summand of $H\circ F$. The induced equivalence relation $\sim_L$ is called \emph{left equivalence}, and the equivalence classes \emph{left cells}. Similarly, we can define the right preorder $\leq_R$ by composing 
 with $H$
from the right, and two-sided preorder $\leq_J$ by composing 
with $H_1$ and $H_2$
from both sides. Right and two-sided equivalence and cells are also defined analogously.

For any transitive 2-representation $\mathbf{M}$ of $\mathscr{C}$, there is, by \cite{CM}, a unique two-sided cell, maximal with respect to the two-sided preorder, which is not annihilated by $\mathbf{M}$.
This two-sided cell is called the \emph{apex} of $\mathbf{M}$.

A two-sided cell $\mathcal{J}$ is called \emph{idempotent} if it contains $F$, $G$ and $H$ such that $H$ is isomorphic to a direct summand of $F\circ G$. The apex of a 2-representation is necessarily idempotent.

Let $\mathcal{L}$ be a left cell in $\mathscr{C}$ and let $\mathtt{i}=\mathtt{i}_\mathcal{L}$ be the object such that all 1-morphisms in $\mathcal{L}$ start at $\mathtt{i}$.
Then the principal representation $\mathbf{P}_\mathtt{i}$ has a subrepresentation given by the additive closure of all 1-morphisms $F$ such that
$F\geq_L \mathcal{L}$. This, in turn, has a unique simple transitive quotient which we call the \emph{cell 2-representation}
associated to $\mathcal{L}$ and denote by $\mathbf{C}_\mathcal{L}$. We refer to \cite{MM1,MM2} for more details.

\subsection{Action matrices}

Let $\mathbf{M}$ be a finitary $2$-representation of $\mathscr{C}$
and $F$ a $1$-mor\-phism in $\mathscr{C}(\mathtt{i},\mathtt{j})$.
Let $X_1,X_2,\dots,X_k$ be a complete list of representatives of
isomorphism classes of indecomposable objects in
$\mathbf{M}(\mathtt{i})$ and $Y_1,Y_2,\dots,Y_m$ be a 
complete list of representatives of isomorphism classes of 
indecomposable objects in $\mathbf{M}(\mathtt{j})$.
Then we can define the {\em action matrix} $[F]$ of $F$ as
the integral $m\times k$-matrix $(r_{ij})_{i=1,\dots,m}^{j=1,\dots,k}$,
where $r_{ij}$ is the multiplicity of $Y_i$ as a direct summand of
$\mathbf{M}(F)X_j$. Clearly, we have $[FG]=[F][G]$.

If $\mathscr{C}$ has only one object, then $\mathbf{M}$ is transitive 
if and only if all coefficients of $[F]$ are positive, where
$F$ is such that it contains, as direct summands, all indecomposable
$1$-morphisms in the apex of $\mathbf{M}$.

If $\overline{\mathbf{M}}(F)$ is exact, then we can also consider 
the matrix $[[F]]$ which bookkeeps the composition
multiplicities of the values of $\overline{\mathbf{M}}(F)$ on simple
objects in $\overline{\mathbf{M}}(\mathtt{i})$.

\section{Bimodules over the dual numbers and the main result}

\subsection{The 2-category of bimodules over the dual numbers}

In the remainder of the paper,
we work over an algebraically closed field $\Bbbk$ of characteristic 0.
Denote by $D=\Bbbk [x]/(x^2)$ the dual numbers. Fix a small category $\mathcal{C}$ equivalent to $D$-mod. Let $\mathscr{D}$ be the 2-category which has
\begin{itemize}
\item one object \texttt{i} (which we identify with $\mathcal{C}$),
\item as 1-morphisms, all endofunctors of $\mathcal{C}$ isomorphic to tensoring with finite dimensional $D$-$D$-bimodules,
\item as 2-morphisms, all natural transformations of functors (these  are given by homomorphisms of the corresponding $D$-$D$-bimodules).
\end{itemize}

Indecomposable $D$-$D$-bimodules can be classified, 
up to isomorphism, following \cite{BR}, \cite{WW}.
Using the notation from \cite{Jo2}, they are the following.
\begin{itemize}
\item The (unique) projective-injective bimodule $D\otimes_{\Bbbk} D$.
\item The band bimodules $B_k(\lambda )$, indexed by
$k\in\mathbb{Z}_{>0}$ and $\lambda\in\Bbbk\setminus\{0\}$. 
The bimodule $B_k(\lambda )$ can be depicted as follows:
\begin{displaymath}
\xymatrix{
\Bbbk^k\ar@/_/[d]_{x\cdot_{-}=\mathrm{Id}}
\ar@/^/[d]^{{}_{-}\cdot x=Q_k(\lambda)}\\
\Bbbk^k
}
\end{displaymath}
where $Q_k(\lambda)$ is the $k\times k$ Jordan cell with 
eigenvalue $\lambda$. In particular, the regular bimodule 
${}_DD_D$ is  isomorphic to the band bimodule $B_1(1)$.
\item String bimodules of four shapes $W$, $S$, $N$ and $M$ indexed by 
$k\in\mathbb{Z}_{\geq 0}$. For a string bimodule $U$, the integer $k$ is the number of \emph{valleys} in the graph representing this bimodule, 
alternatively, $k=\dim (DU\cap UD)$. The graphs representing the 
bimodules $W_1$, $S_1$, $N_1$ and $M_1$ look, respectively, as
follows (here vertices $\bullet$ and $\circ $
represent a fixed basis with $\circ $ depicting the valley, 
the non-zero right action  of $x\in D$ is depicted by horizontal arrows, 
the non-zero left action of  $x\in D$ is depicted by vertical arrows
and all non-zero coefficients of both actions are equal to $1$):
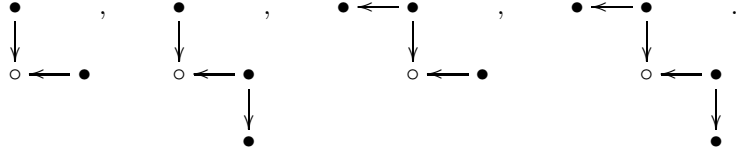
\begin{displaymath}
\xymatrix@=15pt{
\bullet\ar[d]&\\
\circ&\bullet\ar[l]
},\quad\quad
\xymatrix@=15pt{
\bullet\ar[d]&\\
\circ&\bullet\ar[l]\ar[d]\\
&\bullet
},\quad\quad
\xymatrix@=15pt{
\bullet&\bullet\ar[d]\ar[l]&\\
&\circ&\bullet\ar[l]
},\quad\quad
\xymatrix@=15pt{
\bullet&\bullet\ar[d]\ar[l]&\\
&\circ&\bullet\ar[l]\ar[d]\\
&&\bullet
}.
\end{displaymath}
\end{itemize}

An indecomposable bimodule is called $\Bbbk$-split if it is of the form $U\otimes_\Bbbk V$ for indecomposables $U\in D$-mod and $V\in\,\,$mod-$D$. The $\Bbbk$-split bimodules $D\otimes D$, $W_0$, $S_0$ and $N_0$ form the unique maximal two-sided cell
$\mathcal{J}_{\Bbbk \text{-split}}$, 
with left cells inside it  indexed by indecomposable right $D$-modules and right cells inside it indexed by indecomposable left $D$-modules, cf. \cite{MMZ1}.

As was shown in \cite{Jo2}, band bimodules form one cell
(both left, right and two-sided), which we denote $\mathcal{J}_\mathrm{band}$. Moreover, for each positive integer $k$, the four string bimodules with $k$ valleys form a two sided cell $\mathcal{J}_k$, see Section~\ref{section_Jk} for more details. The string bimodule $M_0$ forms its own two-sided cell $\mathcal{J}_0$.
The two-sided cells are linearly ordered as follows:
\begin{equation*}
\mathcal{J}_{\Bbbk \text{-split}}> _J \mathcal{J}_{ M_0}> _J 
\mathcal{J}_1  >_J \mathcal{J}_{2} >_J 
\ldots >_J \mathcal{J}_{\text{band}}.
\end{equation*}
All two-sided cells except $\mathcal{J}_{ M_0}$ are idempotent. Note also that all two-sided cells except the minimal cell $\mathcal{J}_\mathrm{band}$ are finite.

\subsection{ The main result}\label{s-maint}
The following theorem is the main result of this paper.

\begin{theorem}\label{mainthm}
\begin{enumerate}[$($i$)$]
\item Any simple transitive 2-representation of $\mathscr{D}$ with apex $\mathcal{J}_{\Bbbk \text{-split}}$ is equivalent to a cell 2-representation.
\item Any simple transitive 2-representation of $\mathscr{D}$ with apex $\mathcal{J}_k$, where $k\geq 1$, has rank 1 or rank 2. 
\item Any simple transitive 2-representation of $\mathscr{D}$ with apex $\mathcal{J}_k$, where $k\geq 1$, of rank 2 is equivalent to the cell 2-representation $\mathbf{C}_{\mathcal{L}}$, where $\mathcal{L}=\{ M_k,N_k\}$ (or, equivalently, $\mathcal{L}=\{ W_k,S_k\}$).
\item There exists a simple transitive 2-representation of $\mathscr{D}$ with apex $\mathcal{J}_1$ which has rank 1.
\end{enumerate}
\end{theorem}

Taking Theorem~\ref{mainthm} into account,
the following conjecture seems very natural.

\begin{conjecture}
For each $k\geq 1$, there exists a unique,
up to equivalence,
simple transitive 2-representation 
of $\mathscr{D}$  of rank 1 with apex $\mathcal{J}_k$.
\end{conjecture}

\subsection{Proof of Theorem~\ref{mainthm}(i)}
\label{s0p}

For an arbitrary indecomposable $\Bbbk$-split $D$-$D$-bi\-mo\-du\-le
$U\otimes_{\Bbbk}V$, using adjunction and projectivity
of both $V$ and $\mathrm{End}_{D\text{-}}(U)$ as $\Bbbk$-modules, we have 
\begin{equation}\label{eq1.1}
\begin{array}{rcl}
\mathrm{End}_{D\text{-}D}(U\otimes_{\Bbbk}V)&\cong&
\mathrm{Hom}_{D\text{-}D}(U\otimes_{\Bbbk}V,U\otimes_{\Bbbk}V)\\ 
&\cong&\mathrm{Hom}_{\,\text{-}D}(V,
\mathrm{Hom}_{D\text{-}}(U,U\otimes_{\Bbbk}V))\\
&\cong&\mathrm{Hom}_{\,\text{-}D}(V,
\mathrm{Hom}_{D\text{-}}(U,U)\otimes_{\Bbbk}V)\\
&\cong&\mathrm{Hom}_{\Bbbk}(\Bbbk,
\mathrm{Hom}_{\text{-}D}(V,
\mathrm{Hom}_{D\text{-}}(U,U)\otimes_{\Bbbk}V))\\
&\cong&\mathrm{Hom}_{\Bbbk}(\Bbbk,
\mathrm{Hom}_{D\text{-}}(U,U)\otimes_{\Bbbk}
\mathrm{Hom}_{\text{-}D}(V,V))\\
&\cong& \mathrm{End}_{D\text{-}}(U)\otimes_{\Bbbk}
\mathrm{End}_{\text{-}D}(V).
\end{array}
\end{equation}
Consider the finite dimensional 
algebra $A=\End_{D-}({}_DD\oplus {}_D\Bbbk )$ (note that
it can be described as the path algebra of the quiver
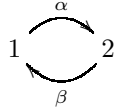
\begin{align*}
\xymatrix{1\ar@/^1pc/[r]^{\alpha} & 2\ar@/^1pc/[l]^{\beta}}
\end{align*}
modulo the relation $\alpha\beta =0$). 
Then we have the $2$-category $\mathscr{C}_A$
of projective $A$-$A$-bimodules. 
By \cite[Theorem 12]{MMZ2}, any simple transitive $2$-representation
of $\mathscr{C}_A$ is equivalent to a cell $2$-representation.

Denote by $\mathscr{A}$ the $2$-full $2$-subcategory
of $\mathscr{D}$ given by the additive closure
inside $\mathscr{D}$ of the regular $D$-$D$-bimodule 
and all $\Bbbk$-split $D$-$D$-bimodules.
The  computation in \eqref{eq1.1} implies that the $2$-categories
$\mathscr{C}_A$ and $\mathscr{A}$ are biequivalent.
Consequently, any simple transitive $2$-representation
of $\mathscr{A}$ is equivalent to a cell $2$-representation.

Let $\mathbf{M}$ be a simple transitive 
$2$-representation of $\mathscr{D}$ with
apex $\mathcal{J}_{\Bbbk \text{-split}}$.
Then the restriction of $\mathbf{M}$ to 
$\mathscr{A}$ is also simple transitive and hence
this restriction is equivalent to a cell
$2$-representation of $\mathscr{A}$ by the previous
paragraph. Now, the arguments similar to the ones in 
\cite[Theorem~18]{MM5} imply that
$\mathbf{M}$ is equivalent to a cell $2$-representation
of $\mathscr{D}$.
This proves Theorem~\ref{mainthm}(i).

\subsection{The two-sided cell $\mathcal{J}_k$, where $k\geq 1$}\label{section_Jk}

Fix a positive integer $k$. Recall from \cite{Jo2} that the two-sided cell $\mathcal{J}_k$ has the following egg-box diagram in which
columns are left cells and rows are right cells.
\begin{align*}
\begin{array}{|c|c|}
\hline
W_k & N_k\\
\hline
S_k& M_k\\
\hline
\end{array}
\end{align*}
Modulo the two-sided cells that are strictly larger with respect to
the two sided order, the multiplication table of
$\mathcal{J}_k$ is as follows.
\begin{equation}\label{eqtable}
\begin{array}{c|c|c|c|c}
\otimes_D & W_k & S_k & N_k & M_k \\
\hline
W_k & W_k & W_k & N_k & N_k \\
\hline
S_k & S_k & S_k & M_k & M_k\\
\hline
N_k & W_k & W_k & N_k & N_k\\
\hline
M_k & S_k & S_k & M_k & M_k \\
\end{array}
\end{equation}

\begin{lemma}\label{adj}
For any $k\geq 0$, the pair $(S_k\otimes _D - ,\, N_k\otimes _D -)$ is an adjoint pair of endofunctors of $D$-mod.
\end{lemma}

\begin{proof}
By \cite[Lemma 13]{MZ2}, it is enough to show that $S_k$ is projective as a left $D$-module, and that $\Hom _{D-}(S_k,D)\simeq N_k$ as $D$-$D$-bimodules. As a left module, $S_k$ is a direct sum  of $k+1$ copies of the left regular module ${}_DD$. This also implies that $\Hom _{D-}(S_k,D)$ is projective as a right module. Moreover
\begin{align*}
\dim \Hom _{D-}(S_k,D) =\dim \Hom _{D-}(D^{\oplus k+1},D)=(k+1)\dim \End _{D-}(D) = 2(k+1) .
\end{align*}
Note that $D$ is a symmetric algebra and 
thus ${}_DD_D\cong {}_DD^*_D$. Hence,
by adjunction, we get
\begin{align*}
\Hom _{D-}(S_k,D) & \simeq  \Hom _{D-}(S_k,\Hom_\Bbbk (D,\Bbbk) ) \simeq   \Hom _{\Bbbk}(D\otimes_D S_k,\Bbbk)  \simeq \Hom _{\Bbbk}(S_k,\Bbbk)
\end{align*}
so that $\Hom _{D-}(S_k,D)$ and $S_k^\ast$ are isomorphic as $D$-$D$-bimodules. 
Since $S_k$ is indecomposable as a $D$-$D$-bimodule, so is $\Hom _{D-}(S_k,D)$.

The indecomposable, right projective, $2(k+1)$-dimensional $D$-$D$-bimodules are:
\begin{itemize}
\item $N_k$,
\item $B_{k+1}(\lambda)$,
\item $D\otimes_{\Bbbk} D$ (in the case $k=1$).
\end{itemize}
To show that $\Hom _D(S_k,D)\simeq N_k$, note first that
\begin{align*}
\Hom_{D-} (S_0,D) = \Hom_{D-} (D\otimes_{\Bbbk}\Bbbk ,D)\simeq \Hom_{\Bbbk} \big( \Bbbk, \Hom_D(D,D) \big) \simeq \Hom_{\Bbbk} (\Bbbk ,D),
\end{align*}
so it is clear that $\Hom_{D-} (S_0,D) \simeq N_0 = \Bbbk \otimes _{\Bbbk} D $, as $D$-$D$-bimodules. Now, for any $k\geq 1$, there is a short exact sequence of $D$-$D$-bimodules
\begin{align*}
0\to S_{k-1} \to S_k \to S_0 \to 0.
\end{align*}
Apply the functor $\Hom_{D-} (- ,D)$ to this sequence. As the regular $D$-$D$-bimodule is injective as a left module, this functor is exact. 
Therefore 
we get a short exact sequence of $D$-$D$-bimodules
\begin{equation}\label{eq:ses}
0\to \Hom_{D-} (S_0 ,D) \to \Hom_{D-} (S_k ,D) \to \Hom_{D-} (S_{k-1} ,D) \to 0.
\end{equation}
Hence $\Hom_{D-} (S_0 ,D)\simeq N_0$ is a submodule of any $\Hom_{D-} (S_k ,D)$, implying that $\Hom_{D-} (S_k ,D)$ is not a band bimodule. This proves the statement for $k\neq 1$. Moreover, by setting $k=2$ in \eqref{eq:ses}, we see that $\Hom_{D-} (S_1 ,D)$ is the quotient of $\Hom_{D-} (S_2 ,D)\simeq N_2$ by $\Hom_{D-} (S_0 ,D)\simeq N_0$, that is
\begin{align*}
\Hom_{D-} (S_1 ,D) \simeq N_2/N_0 \simeq N_1.
\end{align*}
This concludes the proof.
\end{proof}

The following statement is an adjustment of \cite[Theorem~3.1]{Zi2} 
to a slightly more general setting, into which simple transitive 
2-representations of $\mathscr{D}$ with apex $\mathcal{J}_k$, 
where $k\geq 1$, fit.

\begin{theorem}\label{thm_projective}
Let $\mathscr{C}$ be a 2-category with finitely many objects
and such that each $\mathscr{C}(\mathtt{i},\mathtt{j})$ is
$\Bbbk$-linear, idempotent split and has finite dimensional
spaces of $2$-mor\-phisms. 
Let $\mathbf{M}$ be a finitary simple transitive 2-representation of $\mathscr{C}$ such that the apex $\mathcal{J}$ of $\mathbf{M}$ is finite. Assume that $F\in \mathcal{J}$. Then the following holds.
\begin{enumerate}[$($i$)$]
\item For every object $X$ in any $\overline{\mathbf{M}}(\mathtt{i})$, the object $\overline{\mathbf{M}}(F)X$ is projective.
\item If $\overline{\mathbf{M}}(F)$ is left exact, then $\overline{\mathbf{M}}(F)$ is a projective functor.
\end{enumerate}
\end{theorem}

\begin{proof}
We can restrict to the finitary $2$-subcategory of 
$\mathscr{C}$ given by the identities and the apex and
then apply \cite[Theorem~3.1]{Zi2}.
\end{proof}

\begin{cor}
Let $\mathbf{M}$ be a simple transitive 2-representation 
of $\mathscr{D}$ with apex $\mathcal{J}_k$, where $k\geq 1$.
Then the functor $\overline{\mathbf{M}}(N_k)$ is 
a projective functor (in the sense that it is given by tensoring
with a projective bimodule over the underlying algebra
of the 2-representation).
\end{cor}

\begin{proof}
From Lemma~\ref{adj} it follows that 
$\overline{\mathbf{M}}(N_k)$ is left exact. 
Therefore we may apply Theorem~\ref{thm_projective}
and the claim follows.
\end{proof}

\section{Combinatorial results}\label{s1p}

Fix a simple transitive 2-representation \textbf{M} of 
$\mathscr{D}$ with apex $\mathcal{J}_k$, where $k\geq 1$. Let $B$ be a basic associative $\Bbbk$-algebra for which \textbf{M}(\texttt{i}) is equivalent to $B$-proj. Let $1=\varepsilon_1+\ldots +\varepsilon_r$ be a decomposition of the identity in $B$ into a sum of pairwise orthogonal primitive idempotents. Denote by $P_i$ the $i$'th indecomposable projective left $B$-module $B\varepsilon_i$, and denote by $L_i$ its simple top.

The aim of this section is to prove the following.

\begin{prop}\label{prop.matrix}
Let $\mathbf{M}$ be a simple transitive 2-representation of $\mathscr{D}$ with apex $\mathcal{J}_k$, where $k\geq 1$. Then the action matrices of indecomposable 1-morphisms in $\mathcal{J}_k$ are, up to renumbering of projective objects in $\mathbf{M}(\mathtt{i})$, either all equal to $[1]$ or
\begin{align*}
[N_k]=[W_k]=\begin{bmatrix}
1 & 1\\ 0 & 0
\end{bmatrix}, \quad
[M_k]=[S_k]=\begin{bmatrix}
0 & 0\\ 1 & 1
\end{bmatrix}.
\end{align*}
\end{prop}

In particular, Proposition~\ref{prop.matrix} implies Theorem~\ref{mainthm}(ii).
The remainder of this section is devoted to the
proof of Proposition~\ref{prop.matrix}.

\begin{lemma}\label{SN}
\begin{enumerate}[$($i$)$]
\item If the matrix $[N_k]$ has a zero column, then the corresponding row in $[S_k]$ must be zero.
\item If the matrix $[S_k]$ has a zero column, then the corresponding row in $[N_k]$ must be zero.
\end{enumerate}
\end{lemma}

\begin{proof}
By Lemma~\ref{adj}, the functor $N_k$ is exact. 
By \cite[Lemma 10]{MM5}, we have $[[N_k]]=[S_k]^{\mathrm{tr}}$.
If column $i$ in the matrix $[N_k]$ is zero, then $N_kP_i=0$. As $L_i$ is the top of $P_i$,  the object $N_kL_i$ must be zero as well. This proves (i). On the other hand, if column $i$ of $[S_k]$ is zero, then row $i$ of $[S_k]^{\mathrm{tr}}=[[N_k]]$ is zero. This means that nothing in the image of $N_k$ can have $L_i$ as a simple subquotient. In particular, $P_i$ cannot occur in the image of $N_k$, and so row $i$ of $[N_k]$ must be zero. This proves (ii).
\end{proof}

Note that $W_k$, $S_k$, $N_k$ and $M_k$  are all idempotent modulo strictly greater two-sided cells. 
Setting $F=W_k + S_k + N_k + M_k$ yields $F\otimes F=F^{\oplus 4}$. Hence the action matrix of $F$ must be an irreducible positive integer matrix satisfying $[F]^2=4[F]$. The set of such matrices are classified in \cite{TZ}. They are, up to permutations of rows and columns, the following.
\begin{align*}
\begin{bmatrix}
4
\end{bmatrix} ,\
\begin{bmatrix}
2 & 2\\
2 & 2\\
\end{bmatrix} ,\ 
\begin{bmatrix}
3 & 3\\
1 & 1\\
\end{bmatrix} ,\
\begin{bmatrix}
3 & 1\\
3 & 1\\
\end{bmatrix} ,\
\begin{bmatrix}
2 & 4\\
1 & 2\\
\end{bmatrix} ,\
\begin{bmatrix}
2 & 1\\
4 & 2\\
\end{bmatrix} ,\end{align*}
\begin{align*}
\begin{bmatrix}
2 & 1 & 1\\
2 & 1 & 1\\
2 & 1 & 1\\ 
\end{bmatrix} ,\
\begin{bmatrix}
2 & 2 & 2\\
1 & 1 & 1\\
1 & 1 & 1\\ 
\end{bmatrix} ,\
\begin{bmatrix}
1&1&1&1\\
1&1&1&1\\
1&1&1&1\\
1&1&1&1\\
\end{bmatrix}.
\end{align*}
Since the functors $\mathbf{M} (W_k)$, $\mathbf{M}(S_k)$, $\mathbf{M}(N_k)$ and $\mathbf{M}(M_k)$ are all idempotent, their action matrices are
 idempotent
as well. The rank of an idempotent matrix equals 
its
trace. The trace of $[F]$ is 4, so the action matrices $[W_k],[S_k],[N_k],[M_k]$ must all have trace and rank 1. The action matrices also inherit left, right and two-sided preorders and equivalences, so we speak of these notions for 1-morphisms and action matrices interchangeably. Directly from the multiplication table we can also conclude the following about the action matrices.
\begin{itemize}
\item If $A\sim_RB$, then $AB=B$ and $BA=A$. This also implies
\begin{align*}
\im (B)&=\im (AB) \subseteq \im (A) \\
\im (A)&=\im (BA) \subseteq \im (B),
\end{align*}
so that $\im A=\im B$. For matrices of rank 1 this means that all nonzero columns of $A$ and $B$ are linearly dependent.
\item If $A\sim_LB$, then $AB=A$ and $BA=B$. This also implies
\begin{align*}
\ker (A)&=\ker (AB) \supseteq \ker (B) \\
\ker (B)&=\ker (BA) \supseteq \ker (A),
\end{align*}
so that $\ker A=\ker B$. Hence $A$ and $B$ have the same zero columns.
\end{itemize}

\begin{lemma} \label{poss.eq}
\begin{enumerate}[(i)]
\item $[W_k]=[S_k]$ if and only if $[N_k]=[M_k]$.
\item $[W_k]=[N_k]$ if and only if $[S_k]=[M_k]$.
\item If $[W_k]=[M_k]$ or $[S_k]=[N_k]$, then $[W_k]=[S_k]=[N_k]=[M_k]$.
\end{enumerate}
\end{lemma}
\begin{proof}
If $[W_k]=[S_k]$, then
\begin{align*}
[N_k]=[W_k][N_k]=[S_k][N_k]=[M_k].
\end{align*}
On the other hand, if $[N_k]=[M_k]$, then
\begin{align*}
[W_k]=[N_k][W_k]=[M_k][W_k]=[S_k].
\end{align*}
This proves claim $(i)$; claim $(ii)$ is similar.
Finally, if $[W_k]=[M_k]$, then
\begin{align*}
[W_k] =[W_k][W_k] = [W_k][M_k]=[N_k]
\end{align*}
and
\begin{align*}
[W_k] =[W_k][W_k] = [M_k][W_k]=[S_k].
\end{align*}
This proves one of the implications in $(iii)$, the other is similar.
\end{proof}

In particular, Lemma~\ref{poss.eq} implies that, if the matrix $[F]$ has 1 as 
 an entry, then the matrices
$[W_k]$, $[S_k]$, $[N_k]$ and $[M_k]$ are all different.

We now do a case-by-case analysis depending on the rank of the 2-representation (i.e.  the size of action matrices).

\subsection{Rank 1}
If $F=[4]$, then $[W_k] = [S_k] = [N_k] = [M_k]= [1]$.

\subsection{Rank 2}
Consider first the case
$F\in \left\lbrace \begin{bmatrix}
3 & 3\\
1 & 1\\
\end{bmatrix}, \begin{bmatrix}
3 &1\\
3&1\\
\end{bmatrix} \right\rbrace$. Since $F$ has entries equal to 1, the action matrices of $W_k$, $S_k$, $N_k$ and $M_k$ must all be different. They all have trace 1 and their sum has diagonal $(3,1)$, so we must have four different matrices with non-negative integer entries:
\begin{align*}
A=\begin{bmatrix}
1 & a\\
b & 0\\
\end{bmatrix},\
B=\begin{bmatrix}
1 & c\\
d & 0\\
\end{bmatrix},\
C=\begin{bmatrix}
1 & e\\
f & 0\\
\end{bmatrix},\
G=\begin{bmatrix}
0 & g\\
h & 1\\
\end{bmatrix}.
\end{align*}
Two of those with diagonal $(1,0)$, say $A$ and $B$, must belong to the same left cell. Then $AB=A$, i.e.
\begin{align*}
A=\begin{bmatrix}
1 & a\\
b & 0\\
\end{bmatrix} = AB =
\begin{bmatrix}
1+ad & c\\
d & bc\\
\end{bmatrix}
\end{align*}
which implies $a=c$ and $b=d$, so that $A=B$, a contradiction.

Assume now
$F\in \left\lbrace \begin{bmatrix}
2&4\\
1&2\\
\end{bmatrix},\begin{bmatrix}
2&1\\
4&2\\
\end{bmatrix},
\begin{bmatrix}
2&2\\
2&2\\
\end{bmatrix} \right\rbrace $.  Then
$W_k$, $S_k$, $N_k$ and $M_k$ will be given by the following matrices:
\begin{align*}
A=\begin{bmatrix}
1 & \ast\\
\ast & 0\\
\end{bmatrix},\
B=\begin{bmatrix}
1 & \ast\\
\ast & 0\\
\end{bmatrix},\
C=\begin{bmatrix}
0 & \ast\\
\ast & 1\\
\end{bmatrix},\
G=\begin{bmatrix}
0 & \ast\\
\ast & 1\\
\end{bmatrix}.
\end{align*}
We see that $AB,BA\in \{ A,B\}$. This implies that either $A\sim_LB$ or $A\sim_RB$.

If $A\sim_LB$, then $A\not\sim_RB$, so we can assume $A\sim_RC$ and $B\sim_RG$. By comparing images, and using that all ranks are 1, we get
\begin{align*}
A=B=\begin{bmatrix}
1&0\\
1&0\\
\end{bmatrix},\
C=G=\begin{bmatrix}
0&1\\
0&1\\
\end{bmatrix}.
\end{align*}
$A\sim_LB$ and $C\sim_LG$ tells us that left equivalent functors are represented by the same matrix. By symmetry we can set
\begin{align*}
[N_k]=[M_k]=\begin{bmatrix}
1&0\\
1&0\\
\end{bmatrix} \quad \text{and} \quad [S_k]=[W_k]=\begin{bmatrix}
0&1\\
0&1\\
\end{bmatrix}.
\end{align*}
However, now the second column of $[N_k]$ is zero, but the second row of $[S_k]$ is nonzero. This contradicts Lemma~\ref{SN}(i), so we discard this case.

If, instead, $A\sim_RB$ and $C\sim_RG$, we can assume $A\sim_LC$ and $B\sim_LG$. Since the first column of $A$ is nonzero, so is the first column of $C$. At the same time, the second column of $C$ is nonzero, so the second column of $A$ is as well. Together with ranks being 1 and right equivalences, this yields
\begin{align*}
A=B=\begin{bmatrix}
1&1\\
0&0\\
\end{bmatrix}\ \text{and} \
C=G=\begin{bmatrix}
0&0\\
1&1\\
\end{bmatrix}.
\end{align*}
By symmetry we can set
\begin{align*}
[N_k]=[W_k]=\begin{bmatrix}
1&1\\
0&0\\
\end{bmatrix}\ \text{and} \
[M_k]=[S_k]=\begin{bmatrix}
0&0\\
1&1\\
\end{bmatrix}.
\end{align*}

\subsection{Rank 3}
$F\in \left\lbrace \begin{bmatrix}
2 & 1 & 1\\
2 & 1 & 1\\
2 & 1 & 1\\ 
\end{bmatrix} ,\
\begin{bmatrix}
2 & 2 & 2\\
1 & 1 & 1\\
1 & 1 & 1\\ 
\end{bmatrix} \right\rbrace $.
Either choice of the matrix of $F$ has 1 as entry, so all of $[W_k]$, $[S_k]$, $[N_k]$ and $[M_k]$ have to be different. As the diagonal of $F$ is $(2,1,1)$, they must be represented by idempotent matrices $A,B,C,G$, all of rank 1, as follows.
\begin{align*}
A=\begin{bmatrix}
1 && \ast \\
 &0& \\
\ast && 0 \\
\end{bmatrix},\
B=\begin{bmatrix}
1 && \ast \\
 &0& \\
\ast && 0 \\
\end{bmatrix},\
C=\begin{bmatrix}
0 && \ast \\
 &1& \\
\ast && 0 \\
\end{bmatrix},\
G=\begin{bmatrix}
0 && \ast \\
 &0& \\
\ast && 1 \\
\end{bmatrix}.
\end{align*}
Note that $AB,BA\in \{ A,B\}$, so $A$ and $B$ are either left or right equivalent. We consider these two cases separately.

Assume first $A\sim_L B$, $C\sim_LG$, so that $C$ and $G$ have the same kernel. Hence the third column of $C$ and the second of $G$ are nonzero. Since the ranks are 1 we get
\begin{align*}
C=\begin{bmatrix}
0 & & \ast\\
 & 1 & 1\\
\ast & 0 & 0 
\end{bmatrix},\
G=\begin{bmatrix}
0 & & \ast\\
 & 0 & 0\\
\ast & 1 & 1 
\end{bmatrix}.
\end{align*}
Taking into account that the lower  right submatrix of $F$ has all entries 1, this implies
\begin{align*}
A=\begin{bmatrix}
1 & & \ast\\
 & 0 & 0\\
\ast & 0 & 0 
\end{bmatrix},\
B=\begin{bmatrix}
1 & & \ast\\
 & 0 & 0\\
\ast & 0 & 0 
\end{bmatrix}.
\end{align*}
Since $A\sim_LB$,  we have $A\not\sim_RB$. We can thus assume $A\sim_RC$ and $B\sim_RG$. Then $A$ and $C$ have the same image, and $B$ and $G$ have the same image, so that
\begin{align*}
A=\begin{bmatrix}
1 & 0 & 0\\
1 & 0 & 0\\
0 & 0 & 0 
\end{bmatrix},\
B=\begin{bmatrix}
1 & 0 & 0\\
0 & 0 & 0\\
1 & 0 & 0 
\end{bmatrix},\
C=\begin{bmatrix}
0 & 1 & 1\\
0 & 1 & 1\\
0 & 0 & 0 
\end{bmatrix},\
G=\begin{bmatrix}
0 &1  & 1\\
0 & 0 & 0\\
0 & 1 & 1 
\end{bmatrix}.
\end{align*}
Then $\{ S,N\}$ is either $\{ A,G\}$ or $\{ B,C\}$. Any such choice contradicts Lemma~\ref{SN}(ii).

Assume now $A\sim_RB$  and $C\sim_RG$, so that $C$ and $G$ have the same image. Then
\begin{align*}
C=\begin{bmatrix}
0 & & \ast \\
& 1 & 0\\
\ast & 1 & 0
\end{bmatrix},\
G=\begin{bmatrix}
0 & & \ast \\
& 0 & 1\\
\ast & 0 & 1
\end{bmatrix} .
\end{align*}
By considering the lower right $2\times 2$-submatrix of $F$, we conclude
\begin{align*}
A=\begin{bmatrix}
1 & & \ast\\
 & 0 & 0\\
\ast & 0 & 0 
\end{bmatrix},\
B=\begin{bmatrix}
1 & & \ast\\
 & 0 & 0\\
\ast & 0 & 0 
\end{bmatrix}.
\end{align*}
$A\sim_RB$ implies $A\not\sim_LB$, so we can assume $A\sim_LC$ and $B\sim_LG$. Using now  that left equivalence means common kernel, together with all ranks being 1, we get
\begin{align*}
A=\begin{bmatrix}
1 & 1 & 0\\
0 & 0 & 0\\
0 & 0 & 0 
\end{bmatrix},\
B=\begin{bmatrix}
1 & 0 & 1\\
0 & 0 & 0\\
0 & 0 & 0 
\end{bmatrix},\
C=\begin{bmatrix}
0 & 0 & 0 \\
1 & 1 & 0\\
1 & 1 & 0
\end{bmatrix},\
G=\begin{bmatrix}
0 & 0 & 0 \\
1  & 0 & 1\\
1 & 0 & 1
\end{bmatrix} .
\end{align*}
Then $\{ S,N\}$ is either $\{ A,G\}$ or $\{ B,C\}$. Any such choice contradicts Lemma~\ref{SN}(i).

\subsection{Rank 4}
Assume that $W_k$, $S_k$, $N_k$ and $M_k$ are given by
\begin{align*}
A=\begin{bmatrix}
1 &&& \ast \\
& 0 &&  \\
&& 0 & \\
\ast &&& 0
\end{bmatrix},\
B=\begin{bmatrix}
0 &&& \ast \\
& 1 &&  \\
&& 0 & \\
\ast &&& 0
\end{bmatrix},\
C=\begin{bmatrix}
0 &&& \ast \\
& 0 &&  \\
&& 1 & \\
\ast &&& 0
\end{bmatrix},\
G=\begin{bmatrix}
0 &&& \ast \\
& 0 &&  \\
&& 0 & \\
\ast &&& 1
\end{bmatrix}.
\end{align*}
As all entries of $F$ are 1, we have that, 
for each position $(i,j)$, one of $A,B,C,G$ 
has entry 1 at this position, while the others 
have entry 0 at this position. 

We can, without loss of generality, assume 
that $A\sim_RB$, $C\sim_RG$, $A\sim_LC$ and $B\sim_LG$. This gives us immediately
\begin{align*}
A=\begin{bmatrix}
1 & 0&1& 0 \\
1 & 0 &1&0  \\
0&0& 0 &0 \\
0 &0 &0& 0
\end{bmatrix},\
B=\begin{bmatrix}
0 &1 &0& 1 \\
0& 1 &0&1  \\
0&0& 0 &0 \\
0 &0&0& 0
\end{bmatrix},\
C=\begin{bmatrix}
0 &0&0& 0 \\
0& 0 &0& 0 \\
1&0& 1 & 0 \\
1 &0&1& 0
\end{bmatrix},\
G=\begin{bmatrix}
0 &0&0& 0 \\
0& 0 &0& 0 \\
0&1& 0 &1 \\
0 &1&0& 1
\end{bmatrix}.
\end{align*}
Then $\{S_k,N_k\}$ is either $\{ A,G\}$ or $\{ B,C\}$. Any such choice contradicts Lemma~\ref{SN}.
This completes the proof of Proposition~\ref{prop.matrix}.

\section{Each simple transitive $2$-representation of rank $2$ is cell}\label{sec_rank2_cell}

Fix a simple transitive 2-representation \textbf{M} of 
$\mathscr{D}$ with apex $\mathcal{J}_k$, where $k\geq 1$. 

Let $\mathcal{L}$ be the left cell $\{N_k,M_k\}$.
As seen in Proposition~\ref{prop.matrix}, the action matrices of $\mathbf{M}(U_k)$, where $U_k\in \mathcal{J}_k$, are as follows:
\begin{displaymath}
[N_k]=[W_k]=\begin{bmatrix}
1 & 1\\ 0 & 0
\end{bmatrix}, \qquad\qquad
[M_k]=[S_k]=\begin{bmatrix}
0 & 0\\ 1 & 1
\end{bmatrix}.
\end{displaymath}

Let us see what this says about the basic algebra $B$
 underlying $\overline{\mathbf{M}}(\mathtt{i})$. 
The rank is two, so we have a decomposition 
$1=\varepsilon_1+\varepsilon_2$ of 
the identity in $B$
into primitive orthogonal idempotents. Denote by $P_1=B\varepsilon_1$ and $P_2=B\varepsilon_e$ the indecomposable projective left $B$-modules, and by $L_1,L_2$ their respective simple tops. Then, for $i=1,2$,
 we have
\begin{displaymath}
\mathbf{M}(N_k)P_i\simeq \mathbf{M}(W_k)P_i\simeq P_1
\qquad\text{ and }\qquad
\mathbf{M}(S_k)P_i\simeq \mathbf{M}(M_k)P_i\simeq P_2.
\end{displaymath}
Moreover, since
\begin{align*}
[[N_k]]=[S_k]^t=\begin{bmatrix}
0 & 1\\ 0 & 1
\end{bmatrix},
\end{align*}
we have $\overline{\mathbf{M}}(N_k)L_1=0$ and $\overline{\mathbf{M}}(N_k)L_2$ has simple subquotients $L_1, L_2$.
By Theorem~\ref{thm_projective}, it follows that $\overline{\mathbf{M}}(N_k)L_1$ must be isomorphic to a number of copies of $P_1$. Therefore we see that $\overline{\mathbf{M}}(N_k)L_1\simeq P_1$, and $P_1$ has length 2 with socle $L_2$.
In the underlying quiver of $B$ this means that we have exactly one arrow $\alpha$ from 1 to 2, and no loops at 1. If there is an arrow $\beta$ from 2 to 1 then $\beta \alpha =0$. Moreover, if there is a loop $\gamma$ at 2 then  $\gamma\alpha =0$:
\begin{equation}\label{eq22}
\xymatrix{
1\ar@/^1pc/[r]^{\alpha} & 2\ar@/^1pc/@{.>}[l]^{\beta} \ar@(ur,dr)@{.>}^\gamma
}
\end{equation}
This also yields
\begin{align*}
\dim (\varepsilon_1B\varepsilon_1) &= \dim \Hom_B(P_1,P_1) = 1\\
\dim (\varepsilon_2B\varepsilon_1) &= \dim \Hom_B(P_2,P_1) = 1.
\end{align*}
Since $\overline{\mathbf{M}} (N_k)$ is exact and only has $P_1$ in its image, it must be of the form
\begin{align*}
\overline{\mathbf{M}} (N_k)\simeq B\varepsilon_1 \otimes \varepsilon_1 B^{\oplus a} \oplus B\varepsilon_1 \otimes \varepsilon_2 B^{\oplus b}, 
\end{align*}
for some nonnegative integers $a$ and $b$. Since
\begin{align*}
\overline{\mathbf{M}} (N_k) L_1 =0,
\end{align*}
we must have $a=0$. Then 
\begin{align*}
\overline{\mathbf{M}} (N_k) L_2 \simeq P_1
\end{align*}
implies that $b=1$, so that
\begin{align*}
\mathbf{M}(N_k)\simeq B\varepsilon_1\otimes \varepsilon_2B\otimes_B-.
\end{align*}
As seen in Lemma~\ref{adj}, $(S_k,N_k)$ is an adjoint pair, 
so this also gives 
\begin{align*}
\mathbf{M}(S_k)\simeq B\varepsilon_2\otimes (B\varepsilon_1)^{\ast}\otimes_B-,
\end{align*}
cf. \cite[Subsection~7.3]{MM1}.
Again, using that $(S_k,N_k)$ is an adjoint pair, yields
\begin{align*}
\dim (\varepsilon_2B\varepsilon_2) &=\dim \Hom_B (P_2,P_2)=\\
&=\dim \Hom_B (\mathbf{M}(S_k)P_1,P_2) =\\
&=\dim \Hom_B (P_1,\mathbf{M}(N_k)P_2) =\\
&= \dim \Hom_B (P_1,P_1) =\\
&= 1.
\end{align*}
In the quiver \eqref{eq22}, this rules out loops at 2. 
Moreover, it implies
\begin{align*}
\mathbf{M}(W_k) \simeq \mathbf{M}(N_k)\mathbf{M}(S_k) \simeq B\varepsilon_1 \otimes (B\varepsilon_1)^{\ast}\otimes_B-.
\end{align*}
Because $W_k$ is idempotent, $\dim \big( (B\varepsilon_1)^{\ast}\otimes_B B\varepsilon_1\big) =1$. Hence, it follows that
\begin{align*}
\mathbf{M}(M_k)=\mathbf{M}(S_k)\mathbf{M}(N_k)=B\varepsilon_2\otimes \varepsilon_2B\otimes_B-.
\end{align*}
Consider now $(B\varepsilon_1)^{\ast}$.  As seen above, $P_1=B\varepsilon_1$ has Jordan-H{\"o}lder series  $L_1$, $L_2$, so $(B\varepsilon_1)^{\ast}$ has top $L_2^\ast$ and socle $L_1^\ast$ 
(these are simple right $B$-modules). This
implies that $(B\varepsilon_1)^{\ast}$
is exactly the projective right module $\varepsilon_2B$. Hence we conclude
\begin{align*}
\mathbf{M}(N_k)\simeq  \mathbf{M}(W_k)&\simeq B\varepsilon_1\otimes \varepsilon_2B\otimes_B-\\
\mathbf{M}(S_k)\simeq \mathbf{M}(M_k)&\simeq  B\varepsilon_2\otimes \varepsilon_2B\otimes_B-.
\end{align*}

We have that the Cartan matrix of $\mathbf{M}$ is
\begin{align*}
\begin{bmatrix}
1 & c\\
1 & 1
\end{bmatrix}
\end{align*}
where $c=\dim \Hom_B (P_1,P_2)$ remains unknown.

Since $\dim \Hom_B(P_2,P_2)=1$, and $P_1$ has Jordan H\"{o}lder series $L_1,L_2$, we must have a short exact sequence
\begin{align*}
L_1^{\oplus c} \xrightarrow{g} P_2 \to L_2.
\end{align*}
In the quiver \eqref{eq22}, this 
corresponds to the fact that we have exactly $c$ arrows $\beta_1,\ldots ,\beta_c :2\to 1$ and the relations 
\begin{align*}
\alpha\beta_i=0=\beta_i \alpha .
\end{align*}

Let us sum up what we know so far:
\begin{itemize}
\item $P_1$ has basis $\{ \varepsilon_1,\alpha\}$,
\item $P_2$ has basis $\{ \varepsilon_2,\beta_1,\dots,\beta_c\}$,
\item $\mathrm{Hom}_B(P_1,P_2)$ has a basis $\{ f_1,\ldots ,f_c\}$ where $f_i(\alpha )=0$ and $f_i(\varepsilon_1)=\beta_i$.
\end{itemize}
However, all functors above are of the form
\begin{align*}
\overline{\mathbf{M}}(U)= B\varepsilon_i\otimes \varepsilon_2B\otimes_B-.
\end{align*}
The module $\varepsilon_2B$ has basis $\{\varepsilon_2,\alpha\}$, and, as seen above, we have
\begin{align*}
\alpha\beta_i=0=\varepsilon_2\beta_i .
\end{align*}
Thus, for $U\in \mathcal{J}_k$, we have
\begin{align*}
\overline{\mathbf{M}}(U)(f_i)(\varepsilon_1 )=0,
\end{align*}
so that $\overline{\mathbf{M}}(U)(f_i)=0$. But then the $f_i$'s generate a proper $\mathscr{D}$-invariant ideal in $\mathbf{M}(\mathtt{i})$. By simplicity of $\mathbf{M}$, this ideal is $\{ 0\}$. Thus $c=0$ and the Cartan matrix is
\begin{align*}
\begin{bmatrix}
1 & 0\\
1 & 1
\end{bmatrix}.
\end{align*}

The rest of the proof now goes as in e.g. 
\cite[Proposition~9]{MM5}
or \cite[Subsection~4.9]{MaMa}.
Consider the principal 2-representation $\mathbf{P}_\mathtt{i}$ and the subrepresentation $\mathbf{N}$ with $\mathbf{N}(\mathtt{i})=\mathrm{add}\{ F\mid F\geq_L \mathcal{L}\}$. Recall that there is a unique maximal ideal $\mathbf{I}$ in $\mathbf{N}$ such that $\mathbf{N}/\mathbf{I}\simeq \mathbf{C}_\mathcal{L}$. The map
\begin{align*}
\Phi : \mathbf{P}_\mathtt{i}&\to \overline{\mathbf{M}} \\
\mathbbm{1}_\mathtt{i}&\mapsto L_2
\end{align*}
extends to a 2-natural transformation by the Yoneda Lemma, \cite[Lemma 9]{MM2}. Since
\begin{displaymath}
\overline{\mathbf{M}}(N_k)L_2=P_1\quad\text{ and }\quad
\overline{\mathbf{M}}(M_k)L_2=P_2,
\end{displaymath}
$\Phi$ induces a 2-natural transformation $\Psi :\mathbf{N}\to \overline{\mathbf{M}}_\mathrm{proj}$.
Note that $\overline{\mathbf{M}}_\mathrm{proj}$ is equivalent to $\mathbf{M}$. By uniqueness of the maximal ideal $\mathbf{I}$ the kernel of $\Psi$ is contained in $\mathbf{I}$, so $\Psi$ factors through  $\mathbf{C}_\mathcal{L}$. On the other hand, the Cartan matrices of $\mathbf{M}$ and $\mathbf{C}_\mathcal{L}$ coincide. Consequently, $\Psi$ induces an equivalence of 2-representations between $\mathbf{C}_\mathcal{L}$ and $\mathbf{M}$.

This proves Theorem~\ref{mainthm}(iii).

\section{A simple transitive $2$-representation 
of rank $1$ with apex $\mathcal{J}_1$}\label{s3p}

Recall that we have the two-sided cell $\mathcal{J}_{0}$ containing only the 1-morphism $M_0$. We have $\mathcal{J}_{\Bbbk\text{-split}}\geq_J\mathcal{J}_0\geq_J\mathcal{J}_1$. The cell $\mathcal{J}_0$ is not idempotent, since
\begin{align*}
M_0\otimes_D M_0 \simeq D\otimes D \oplus \Bbbk .
\end{align*}
However, for all $U\in \mathcal{J}_1$, we have
\begin{align*}
U\otimes _DM_0\simeq M_0 \oplus V,
\end{align*}
where all indecomposable direct summands of $V$ are $\Bbbk$-split. Since $\mathcal{J}_0$ contains only one element, it is also a left cell. Therefore the cell 2-representation $\mathbf{C}_{\mathcal{J}_0}$ is a simple transitive 2-representation of $\mathscr{D}$ with apex $\mathcal{J}_1$. Note that the matrix describing the action of each 1-morphism in $\mathcal{J}_1$ is $[1]$, agreeing with  Proposition~\ref{prop.matrix}.

This proves Theorem~\ref{mainthm}(iv)
and thus completes the proof of Theorem~\ref{mainthm}.

\section{(Co-) Duflo 1-morphisms}\label{sdufloo}

\subsection{2-morphisms to and from $\mathbbm{1}_\mathtt{i}$}\label{sectionhom}

\subsubsection{String bimodules}\label{s_hom_string}

In what follows we will need a more detailed description of
string bimodules. We will index the basis elements of 
$M_k$ and $W_k$ as follows: 
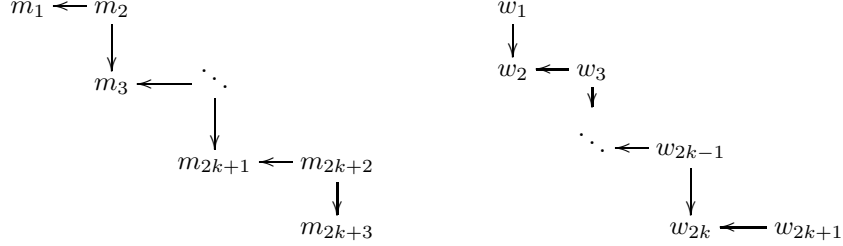
\begin{displaymath}
\xymatrix@=12pt{
m_1 & m_2\ar[l]\ar[d]\\
& m_3 &  \ar[l]\ddots \ar[d]\\
&& m_{2k+1} & m_{2k+2}\ar[l]\ar[d] \\
&&& m_{2k+3}}\qquad\qquad
\xymatrix@=12pt{
w_1\ar[d] \\
w_2 & w_3\ar[l]\ar[d]\\
&{\ddots} &w_{2k-1} \ar[l] \ar[d]\\
&&w_{2k} & w_{2k+1}\ar[l]
}
\end{displaymath}
With this convention, we have $N_k\simeq M_k/ \mathrm{span}\{ m_{2k+3} \}$, $S_k\simeq M_k/\mathrm{span}\{ m_1 \}$ and $W_k\simeq M_k/ \mathrm{span}\{ m_1,m_{2k+3} \}$.

\begin{lemma}\label{lemhomto1}
Let $k$ be a positive integer.
\begin{enumerate}[$($i$)$]
\item The only element of $\mathcal{J}_k$ admitting a $D$-$D$-bimodule morphism to $\mathbbm{1}_\mathtt{i}$ which does not factor through
the simple bimodule is $M_k$.
\item The only element of $\mathcal{J}_k$ admitting a $D$-$D$-bimodule morphism $\mathbbm{1}_\mathtt{i}\to U_k$ which does not factor 
through the simple bimodule is $W_k$.
\end{enumerate}
\end{lemma}

\begin{proof}
The regular bimodule
$\mathbbm{1}_\mathtt{i}\simeq {}_DD_D$ has 
standard
basis $\{ 1,x\}$.

There is a $D$-$D$-bimodule morphism $\varphi_k :M_k\to \mathbbm{1}_\mathtt{i}$ given by
\begin{align*}
\varphi_k (m_j)=\begin{cases} 1, & j\ \text{even} \\ x,& j\ \text{odd}\end{cases} .
\end{align*}
That is, $\varphi_k$ maps standard basis elements from $\mathrm{rad}(M_k)$ to $x\in {}_DD_D$, and the rest of the standard basis elements to 1. We prove that any  $D$-$D$-bimodule morphism $\varphi :W_k\to \mathbbm{1}_\mathtt{i}$ factors through the simple bimodule, and similar arguments for $S_k$ and $N_k$ complete the proof of part (i). 
Assume that $\varphi :W_k\to \mathbbm{1}_\mathtt{i}$ is a $D$-$D$-bimodule morphism. 
Consider the standard basis vector $w_1$. Since $w_1x=0$ we must have $\varphi (w_1)\in \mathrm{span}\{ x\}$. Thus
\begin{align*}
\varphi (w_2)=\varphi (xw_1) = x\varphi(w_1)=0.
\end{align*}
As $w_2=w_3x$ this, in turn, implies $\varphi (w_3)\in \mathrm{span}\{ x\}$ and so on. We will have $\varphi (w_{2j})=0$ for all $j$, i.e. $\varphi$ annihilates $\mathrm{rad}(W_k)$. Thus $\varphi$ factors through the simple bimodule.

For  part (ii),  it is straightforward to check that 
$\psi_k:\mathbbm{1}_\mathtt{i}\to W_k$ given by
\begin{align*}
\psi_k (1)&= w_1+w_3+\ldots +w_{2k+1},\\
\psi_k (x)&= w_2+w_4+\ldots +w_{2k},
\end{align*}
is a homomorphism of $D$-$D$-bimodules.
If $\eta :\mathbbm{1}_\mathtt{i}\to M_k$ is  a bimodule morphism, then
\begin{align*}
\eta (1)=\sum_{j=1}^{2k+3} \lambda_j m_j,
\end{align*}
for some $\lambda_j\in \Bbbk$. Then
\begin{displaymath}
\eta (x) = \eta (1)x = \sum_{j=1}^{k+1} \lambda_{2j} m_{2j-1} 
=x\eta (1) =  \sum_{j=1}^{k+1} \lambda_{2j} m_{2j+1} .
\end{displaymath}
Comparing the coefficients, we conclude that $\lambda_{2j}=0$, for $j=1,\ldots ,k+1$, so that $\eta (1)\in \mathrm{rad}(M_k)$. Thus $\varphi$ factors through the simple bimodule. For $S_k$ and $N_k$,
the proof is similar.
\end{proof}

Note that $\varphi_0:M_0\to \mathbbm{1}_\mathtt{i}$
is also defined.
If  we fix integers $l\leq k$, then $\varphi_l$ factors through $\varphi_k$, and $\psi_l$ factors through $\psi_k$. Indeed,  $M_k$ has a submodule isomorphic to $M_l$ spanned by $\{ m_j\mid j=1,\ldots ,2l+3\}$.  Letting $\iota_{l,k}:M_l\to M_k$ be the inclusion of $M_l$ into $M_k$, it is clear that $\varphi_l=\varphi_k\circ \iota_{l,k}$. Similarly, denote by $\pi_{k,l}:W_k\to W_l$ the projection whose kernel is spanned by $\{ w_j\mid j\geq 2l+2\}$. Then $\psi_l=\pi_{k,l}\circ \psi_k$.

Let us now address the problem of uniqueness of $\varphi_k$ and $\psi_k$.
For a non-negative integer $k$, denote by $V_k$ the subspace
of $\mathrm{Hom}_{D\text{-}D}(M_k,D)$
consisting of all homomorphisms which factor through the
simple $D$-$D$-bimodule. 
For a positive integer $k$, denote by $\hat{V}_k$ the subspace
of $\mathrm{Hom}_{D\text{-}D}(D,W_k)$
consisting of all homomorphisms which factor through the
simple $D$-$D$-bimodule. 

\begin{cor}\label{cordufcor}
{\hspace{1mm}}

\begin{enumerate}[$($i$)$]
\item For any non-negative integer $k$, we have
$\dim\mathrm{Hom}_{D\text{-}D}(M_k,D)/V_k=1$.
\item For any positive integer $k$, we have
$\dim\mathrm{Hom}_{D\text{-}D}(D,W_k)/\hat{V}_k=1$.
\end{enumerate}
\end{cor}

\begin{proof}
Assume that $\varphi\in \mathrm{Hom}_{D\text{-}D}(M_k,D)\setminus V_k$.
Then $\varphi(m_2)\in D\setminus\Bbbk\langle x\rangle$, in particular,
$x\varphi(m_2)=\varphi(xm_2)=\varphi(m_3)\neq 0$.
Using the right action of $x$, we have 
$\varphi(m_3)=\varphi(m_4x)=\varphi(m_4)x$, which uniquely determines
the image of $\varphi(m_4)$ in $D/\Bbbk\langle x\rangle$.
Similarly, the image of each $\varphi(m_i)$, where $i$ is even, in 
$D/\Bbbk\langle x\rangle$ is uniquely determined. As 
$\Bbbk\langle x\rangle\subset D$ is a simple $D$-$D$-bimodule,
claim (i) follows. Claim (ii) is proved similarly.
\end{proof}

\subsubsection{Band bimodules}

From the definition of band bimodules, it follows
directly that, for all $n\geq 2$, there are short exact sequences of $D$-$D$-bimodules
\begin{align*}
0\to B_1(1)\xrightarrow{\alpha_n} B_n(1) \to B_{n-1}(1)\to 0
\end{align*}
and
\begin{align*}
0\to B_{n-1}(1)\to B_n(1)\xrightarrow{\beta_n} B_1(1) \to 0.
\end{align*}
It is a technical but not difficult exercise to verify that,
for any $n$ and $k$,  the morphism $\varphi_k$ factors through $\beta_n$, and the morphism $\alpha_n$ factors through $\psi_k$.

\subsection{Duflo 1-morphisms in fiat $2$-categories}

Following \cite{MM1}, recall that a finitary $2$-category  
$\mathscr{C}$ is called {\em fiat} if it has a weak
involution $\star$ such that each pair $(F,F^{\star})$
of $1$-morphisms is an adjoint pair via some choice
of adjunctions morphisms between the compositions $FF^{\star}$,
$F^{\star}F$ and the relevant identities.

Let $\mathscr{C}$ be a fiat 2-category and $\mathcal{L}$ a 
left cell in $\mathscr{C}$. Let $\mathtt{i}=\mathtt{i}_\mathcal{L}$ 
be the object such that all 1-morphisms in $\mathcal{L}$ start 
in $\mathtt{i}$. A 1-morphism $G\in \mathcal{L}$ is called a 
{\em Duflo 1-morphism}  for $\mathcal{L}$,
cf. \cite[Subsection~4.5]{MM1}, if the indecomposable 
projective module $P_{\mathbbm{1}_\mathtt{i}}$ in 
$\overline{\mathbb{P}}_\mathtt{i}(\mathtt{i})$ has a submodule 
$K$ such that
\begin{enumerate}
\item $P_{\mathbbm{1}_\mathtt{i}}/K$ is annihilated by all $F\in \mathcal{L}$,
\item there is a surjective morphism $P_G\to K$.
\end{enumerate}
By \cite[Proposition 17]{MM1}, any left cell in a fiat 2-category $\mathscr{C}$ has a unique Duflo 1-morphism. These Duflo 1-morphisms 
play a major role in the construction of cell 2-representations,
cf. \cite{MM1}. 

\subsection{Duflo 1-morphisms for finitary $2$-categories}
\label{s1-duf}

The paper \cite{Zh} gives a different definition of the notion
of Duflo $1$-morphisms which is also applicable for 
general finitary $2$-categories. One significant difference
with \cite{MM1} is that,
in the general case, Duflo $1$-morphisms in the sense of
\cite{Zh} do not have to exist, and if they exist, they do not have
to belong to the left cell they are associated to.
Below we propose yet another alternative.

Let $\mathscr{C}$ be a finitary $2$-category, 
$\mathcal{L}$ a left cell in $\mathscr{C}$
and $\mathtt{i}=\mathtt{i}_\mathcal{L}$ the object such
that all $1$-morphisms in $\mathcal{L}$ start at it.

\begin{definition}\label{def_great}
{\hspace{1mm}}

\begin{enumerate}[(i)]
\item A 1-morphism $G$ in $\mathscr{C}$ is \emph{good} for $\mathcal{L}$ if there is a 2-morphism $\varphi :G\to \mathbbm{1}_\mathtt{i}$ such that $F\varphi :FG\to F$ is right split, for any $F\in \mathcal{L}$
 (i.e. there is $\xi:F\to FG$ such that 
$F\varphi\circ_v\xi=\mathrm{id}_F$).
\item A 1-morphism $G$ in $\mathscr{C}$ is  \emph{great} for $\mathcal{L}$ if it is good for $\mathcal{L}$, and, for any $G'$ with $\varphi ':G'\to \mathbbm{1}_\mathtt{i}$ which is also good for $\mathcal{L}$, there is a 2-morphism $\beta :G\to G'$ such that $\varphi =\varphi '\circ \beta$.
\end{enumerate}
\end{definition}

\begin{rmk}
That $\mathscr{C}$ is finitary is not necessary for to state Definition~\ref{def_great}.
\end{rmk}

For fiat $2$-categories, the following proposition
relates the latter notion to that of Duflo $1$-morphisms.

\begin{prop}\label{pduf-fin}
Let $\mathscr{C}$ be a fiat 2-category and $\mathcal{L}$ a left cell in $\mathscr{C}$. Then $G\in \mathcal{L}$ is great for $\mathcal{L}$ if and only if $G$ is the Duflo 1-morphism of $\mathcal{L}$.
\end{prop}

\begin{proof}
The proof goes as follows: we first prove that the Duflo 1-morphism of $\mathcal{L}$ is good for $\mathcal{L}$. Then we prove that if $G$ is great for $\mathcal{L}$, then $G$ is the Duflo 1-morphism for $\mathcal{L}$. Finally, we prove that the Duflo 1-morphism is great for $\mathcal{L}$.

Assume first that $G$ is the Duflo 1-morphism of $\mathcal{L}$. Let $K\subseteq P_{\mathbbm{1}_\mathtt{i}}$ be the submodule from the definition and $\alpha :P_G\to K$ a surjective morphism. Let $f:P_G\to P_{\mathbbm{1}_\mathtt{i}}$ be the composition $P_G\xrightarrow{\alpha} K\overset{\iota}{\hookrightarrow} P_{\mathbbm{1}_\mathtt{i}}$. 
The morphism $f$ 
is given by a morphism $\varphi :G\to \mathbbm{1}_\mathtt{i}$ 
as represented on the commutative diagram
\begin{equation}\label{eq33}
\xymatrix{
P_G\ar[d]^{f}&=&0\ar[d]\ar[r] & G\ar[d]^{\varphi} \\
P_{\mathbbm{1}_\mathtt{i}}&=&0\ar[r] & \mathbbm{1}_\mathtt{i}.
}
\end{equation}
Consider short exact sequences
\begin{displaymath}
\begin{tikzcd}
\ker\arrow[hook]{r} & P_G \arrow[->>, "\alpha"]{r} & K \\
K\arrow[hook, "\iota"]{r} & P_{\mathbbm{1}_\mathtt{i}}\arrow[->>]{r} & P_{\mathbbm{1}_\mathtt{i}}/K.
\end{tikzcd} 
\end{displaymath}
As $\mathscr{C}$ is fiat, each 
$1$-morphism of $\mathscr{C}$ acts as an exact functor on
each abelian $2$-representation of $\mathscr{C}$. Therefore
applying $F\in \mathcal{L}$ yields short exact sequences
\begin{displaymath}
\begin{tikzcd}
F\ker\arrow[hook]{r} & FP_G \arrow[->>,"F\alpha"]{r} & FK \\
FK\arrow[hook,"F\iota"]{r} & FP_{\mathbbm{1}_\mathtt{i}}\arrow[->>]{r} & F\big( P_{\mathbbm{1}_\mathtt{i}}/K \big).
\end{tikzcd}
\end{displaymath}
By assumption $F\big( P_{\mathbbm{1}_\mathtt{i}}/K \big)=0$, so $F\iota : FK\to FP_{\mathbbm{1}_\mathtt{i}}$ is an isomorphism, in particular, it is surjective. Thus $Ff= F\iota \circ F\alpha :FP_G\to F\mathbbm{1}_\mathtt{i}$ is also surjective, implying that it is right split.

By considering the right column of the diagram
\begin{align*}
\xymatrix{
0\ar[r]\ar[d] & FG \ar[d]^{F\varphi}\\
0\ar[r] & F,
}
\end{align*}
we see that $F\varphi$ is right split. Therefore $G$ is good for $\mathcal{L}$.
This completes the first step of our proof.

To prove the second step, assume that $G$ is great for $\mathcal{L}$. 
Let $\varphi :G\to \mathbbm{1}_\mathtt{i}$ be the corresponding
$2$-morphisms from the definition. This extends to a morphism $P_G\to P_{\mathbbm{1}_i}$ in $\overline{\mathbf{P}}_\mathtt{i}$ as
in \eqref{eq33}
and the submodule $K$ of $P_{\mathbbm{1}_i}$ is the image of this morphism. We now have a short exact sequence
\begin{displaymath}
\xymatrix{
0\to K\xrightarrow{\overline{f}} P_{\mathbbm{1}_\mathtt{i}} \xrightarrow{g} P_{\mathbbm{1}_\mathtt{i}} /K\to 0.
}
\end{displaymath}
Applying exact $F\in \mathcal{L}$, we get a short exact sequence
\begin{displaymath}
\xymatrix{
0\to FK\xrightarrow{F\overline{f}} P_F \xrightarrow{Fg} F\big( P_{\mathbbm{1}_\mathtt{i}} /K \big) \to 0.
}
\end{displaymath}
Note that, since $F\varphi$ is right split, the induced morphism $K\to P_F$ in $\overline{\mathbf{P}}_\mathtt{i}$ is also right split and therefore surjective. Hence $F\overline{f}:FK\to P_F$ is an isomorphism.
By exactness, we obtain  $F\big( P_{\mathbbm{1}_\mathtt{i}} /K \big)=0$.

To conclude that $G$ is the Duflo $1$-morphism
in $\mathcal{L}$, it remains to show that $G\in \mathcal{L}$. Assume that $H$ is the Duflo 1-morphism of $\mathcal{L}$, and that $K_H\subseteq P_{\mathbbm{1}_\mathtt{i}}$ is the submodule from the definition. We shall prove that $G=H$. By the above, $H$ is good for $\mathcal{L}$,
with the corresponding morphism 
$H\to \mathbbm{1}_\mathtt{i}$,
so there is a morphism $\alpha:G\to H$ making the following diagram commutative:
\begin{align*}
\xymatrix{
H\ar[r] & \mathbbm{1}_\mathtt{i} \\
G\ar[ur]_{\varphi}\ar[u]^{\alpha} 
}
\end{align*}
Therefore $K \subseteq K_H\subseteq P_{\mathbbm{1}_\mathtt{i}}$. Note that $K_H$ has simple top $L_H$. By \cite[Proposition 17(b)]{MM1}, for all $F\in \mathcal{L}$, the object $FL_{H}$ has simple top $L_F$, in particular $FL_H\neq 0$. Since $F(P_{\mathbbm{1}_\mathtt{i}}/K)=0$, for all $F\in \mathcal{L}$, we conclude that $K_H\subseteq K$. Thus $K_H=K$. But $K_H$ has simple top $L_H$ and $K$ has simple top $L_G$, so $H=G$ is the Duflo 1-morphism of $\mathcal{L}$.
This completes the second step of our proof.

Finally, let $G$ be the Duflo 1-morphism of $\mathcal{L}$. We have already seen that $G$ is good for $\mathcal{L}$, it remains to prove that it is great. Assume that $H$ is also good for $\mathcal{L}$, with $\psi :H\to \mathbbm{1}_\mathtt{i}$ being
the morphism such that $F\psi$ is right split, for all $F\in \mathcal{L}$. 

As above, $\im \varphi$ and $\im \psi$ give submodules $K_G$ and $K_H$ of $P_{\mathbbm{1}_\texttt{i}}$ with $F(P_{\mathbbm{1}_\texttt{i}}/K_G)=0$ and $F(P_{\mathbbm{1}_\texttt{i}}/K_H)=0$, for all $F\in\mathcal{L}$. Since the top of $K_G$ is $L_G$ and $L_G$ is not annihilated by $F\in \mathcal{L}$, there is a nonzero morphism $K_G\to K_H$ such that the diagram
\begin{displaymath}
\xymatrix{
P_G\ar@{->>}[r] & K_G\ar[d] \ar[r] & P_{\mathbbm{1}_\mathtt{i}}\ar@{=}[d] \\
P_H\ar@{->>}[r] & K_H\ar[r] & P_{\mathbbm{1}_\mathtt{i}}
}
\end{displaymath}
commutes.
Since $P_G$ is projective, there is a morphism $\alpha :P_G\to P_H$ making the left square commute. Thus the whole diagram commutes
and we obtain a factorization
\begin{displaymath}
\xymatrix{
P_G \ar[r]\ar[d]^{\alpha} & P_{\mathbbm{1}_\mathtt{i}} \\
P_H\ar[ur]
}
\end{displaymath}
implying that $G$ is great for $\mathcal{L}$.
\end{proof}

\subsection{Duflo 1-morphisms in $\mathscr{D}$}\label{subs_duflo}

For a positive integer $m$, we denote by 
$\mathscr{D}^{(m)}$ the $2$-full $2$-subcategory of
$\mathscr{D}$ given by the additive closure of
all $1$-morphisms in all two-sided cells $\mathcal{J}$
such that $\mathcal{J}\geq_J\mathcal{J}_m$, together
with $\mathbbm{1}_{\mathtt{i}}$. Note that 
$\mathscr{D}^{(m)}$ is a finitary $2$-category.

The following proposition suggests that 
$M_k$ is a very good candidate for being called a
{\em Duflo $1$-morphism} in its left cell in $\mathscr{D}^{(k)}$.

\begin{prop}\label{propduflo-1}
For any $m\geq k\geq 1$, the $1$-morphism 
$M_k$ of the finitary $2$-category $\mathscr{D}^{(m)}$
is great for $\mathcal{L}=\{ N_k,M_k\}$.
\end{prop}

\begin{proof}
Let us first establish that $M_k$ is good for $\mathcal{L}$.
It is easy to check, by a direct computation
(see Subsection~\ref{s8.2}), that the 
composition $M_k\otimes M_k$ has a direct summand 
isomorphic to $M_k$ spanned by
\begin{align*}
\{ m_2\otimes m_1,\ m_j\otimes m_{j},\ m_{j+1}\otimes m_j 
\mid j=2,4,\ldots ,2k+2\} ,
\end{align*}
and that the projection onto this summand is a right inverse to 
$M_k\varphi_k$. Using $N_k\simeq M_k/\mathrm{span} \{ m_{2k+3}\}$, 
gives also that $N_k\varphi_k$ is right split. 
Therefore $M_k$ is good for $\mathcal{L}$ with respect to the
morphism $\varphi_k:M_k\to \mathbbm{1}_{\mathtt{i}}$.
Note also that, by Corollary~\ref{cordufcor},
the choice of $\varphi_k$ is unique up to a non-zero scalar
and up to homomorphisms which factor through the simple
$D$-$D$-bimodule.

Let now $F$ be a $1$-morphism in $\mathscr{D}^{(k)}$
which is good for $\mathcal{L}$ via the map $\alpha:F\to D$.
To start with, we argue that $\alpha$ does not factor through
the simple $D$-$D$-bimodule. Indeed, if $\alpha$ does factor
through the simple $D$-$D$-bimodule, it is not surjective
as a map of $D$-$D$-bimodules. Applying the right exact
functor $M_k\otimes_D{}_-$ to the exact sequence
\begin{displaymath}
F\overset{\alpha}{\longrightarrow} D
\longrightarrow \mathrm{Coker}\to 0,
\end{displaymath}
we get the exact sequence
\begin{displaymath}
M_k\otimes_D F\overset{M_k\otimes_D\alpha}{\longrightarrow} M_k
\longrightarrow M_k\otimes_D \mathrm{Coker}\to 0.
\end{displaymath}
Note that  $\mathrm{Coker}$ is the simple 
$D$-$D$-bimodule and that $M_k\otimes_D \mathrm{Coker}\neq 0$.
Therefore $M_k\otimes_D\alpha$ is not right split.
This implies that $\alpha$ is surjective 
as a map of $D$-$D$-bimodules.

Now we show that if $F$ has an indecomposable direct summand $G\in \mathcal{J}_l$, $k<l\leq m$, 
such that the restriction of $\alpha$ to $G$ does not factor through the simple $D$-$D$-bimodule, 
then $\varphi_k$ factors through $\alpha$. 
Indeed, by Lemma~\ref{lemhomto1} the only such possibility is $G\simeq M_l$,
and by Corollary~\ref{cordufcor}
the restriction of $\alpha$ to this summand is a scalar multiple of $\varphi_l$.
As noted in Section~\ref{s_hom_string}, $\varphi_k$ factors via $\varphi_l$ for $k\leq l$,
so this provides a factorization of $\varphi_k$ through $\alpha$.
 
As the next step, we show that if the condition of the previous paragraph is not satisfied,
then $F$ contains a summand 
isomorphic to either $D$ or $M_k$ such that the 
restriction of $\alpha$ to this summand does not 
factor through the simple $D$-$D$-bimodule. 
Indeed, assume that this is not the case.
Then, by Lemma~\ref{lemhomto1}, the only possible indecomposable
summands $G$ of $F$ for which the restriction of $\alpha$
does not factor through the simple $D$-$D$-bimodule
come from two-sided cells
$\mathcal{J}$ such that $\mathcal{J}>_J\mathcal{J}_k$.
However, for such $G$, the composition
$M_kG$ cannot have any summands in $\mathcal{J}_k$
since $\mathcal{J}>_J\mathcal{J}_k$. Since $M_k$
is indecomposable, it follows that any morphism $M_kG\to M_k$
is a radical morphism. That $M_kG\to M_k$ is  a radical
morphism, for any summand $G$ isomorphic to 
$D$ or $M_k$, follows from our assumption by the arguments
in the previous paragraph. Therefore $M_k\otimes_D\alpha$
is a radical morphism and hence not right split,
as $M_k$ is indecomposable, a contradiction. 

Because of the previous paragraph, there is a direct
summand $G$ of $F$ isomorphic to either $M_k$ or $D$
such that the restriction of $\alpha$ to $G$
does not factor through the simple $D$-$D$-bimodule.
If $G\cong D$, then the restriction of
$\alpha$ to it is an isomorphism. We can pull back
$\varphi_k$ via this isomorphism and define the
map from $M_k$ to all other summands of $F$ as zero.
This provides the necessary factorization of 
$\varphi_k$ via $\alpha$. 

If $G\cong M_k$, we can pull back
$\varphi_k$ using first Corollary~\ref{cordufcor}
and then correction via morphisms from $M_k$
to the socle of $G$ (such morphisms 
factor through the simple $D$-$D$-bimodule).
In any case, the constructed factorization
implies that $M_k$ is great for $\mathcal{L}$
and completes the proof of our proposition.
\end{proof}

\subsection{Co-Duflo 1-morphisms in $\mathscr{D}$}
\label{s2-duf}

We can dualize Definition~\ref{def_great}. Given a 2-category $\mathscr{C}$ and a left cell $\mathcal{L}$ in $\mathscr{C}$ with $\mathtt{i}=\mathtt{i}_\mathcal{L}$, we say that a 1-morphism $H$ in $\mathscr{C}$ is \emph{co-good} for $\mathcal{L}$ if there is a 2-morphism $\psi: \mathbbm{1}_\mathtt{i}\to H$ such that $F\psi$ is left split, for all $F\in \mathcal{L}$. Moreover, we say that $H$ is \emph{co-great} for $\mathcal{L}$ if $H$ is co-good for $\mathcal{L}$ and, for any $H'$ which is co-good for $\mathcal{L}$ with $\psi': \mathbbm{1}_\mathtt{i}\to H'$, there is a 2-morphism $\gamma :H'\to H$ such that $\psi = \gamma \circ \psi '$.

The following proposition suggests that 
$W_k$ is a very good candidate for being called a
{\em co-Duflo $1$-morphism} in its left cell in $\mathscr{D}^{(k)}$.

\begin{prop}
For any $m\geq k\geq 1$, the $1$-morphism $W_k$ of the finitary 
$2$-category $\mathscr{D}^{(m)}$ is co-great for 
the left cell $\mathcal{L}=\{ W_k,S_k\}$.
\end{prop}

\begin{proof}
Consider the 2-morphism $\psi_k:\mathbbm{1}_\mathtt{i}\to W_k$. 
By a direct calculation, it is easy to check that $W_k\otimes W_k$ 
has a unique direct summand isomorphic to $W_k$ and that
$S_k\otimes W_k$ has a unique direct summand isomorphic to $S_k$.
The projections onto these summands provide left inverses for 
$W_k\psi_k$ and $S_k\psi_k$, respectively. This implies
that $W_k$ is co-good for $\mathcal{L}$ via $\psi_k$.

Assume now that $F$ is co-good for $\mathcal{L}$ via
some $\alpha:\mathbbm{1}_\mathtt{i}\to F$.
We need to construct a factorization $F\to W_k$.
Since multiplication with $x$ is a nilpotent endomorphism
of $D$, the endomorphism $W_k\otimes_A x$ is a nilpotent
endomorphism of $W_k$. In particular, this endomorphism
is a radical map. By a direct computation, one can check
that, for any $\beta:D\to W_k$ which factors through the
simple $D$-$D$-bimodule, the endomorphism $W_k\otimes_A \beta$
is not injective, in particular, it is a radical map.

Now, using arguments similar to the ones in the proof of
Proposition~\ref{propduflo-1}, one shows that there must
exist a summand $G$ of $F$, the restriction of $\alpha$
to which does not factor through the simple $D$-$D$-bimodule
and that this summand must be isomorphic to either 
$D$ or $W_l$ for some $l\geq k$. In the former case, the restriction of
$\alpha$ to $G$ is an isomorphism and the necessary 
factorization $F\to W_k$ is constructed 
via $G\to W_k$ using this isomorphism.
In the latter case, the necessary factorization  is
constructed via $G\to W_k$ using Corollary~\ref{cordufcor} and
the observation that $\varphi_l$ factors via $\varphi_k$ for $k<l$,
and then correction via morphisms from $G$
to $W_k$ which factor through the simple $D$-$D$-bimodule.
\end{proof}

\section{Some algebra and coalgebra 1-morphisms in $\mathscr{D}$}
\label{scoalgg}

\subsection{Algebra and coalgebra 1-morphisms}

Let $\mathscr{C}$ be a $2$-category. Recall that an {\em algebra 
structure} on a $1$-morphism $A\in\mathscr{C}(\mathtt{i},\mathtt{i})$
is a pair $(\mu,\eta)$ of morphisms $\mu:AA\to A$ and $\eta: 
\mathbbm{1}_{\mathtt{i}}\to A$ which satisfy the usual associativity and
unitality axioms
\begin{displaymath}
\mu\circ_v(\mu\circ_h\mathrm{id})=
\mu\circ_v(\mathrm{id}\circ_h\mu),\quad
\mathrm{id}=\mu\circ_v(\mathrm{id}\circ_h\eta),\quad
\mathrm{id}=\mu\circ_v(\eta\circ_h\mathrm{id}).
\end{displaymath}
Similarly, a {\em coalgebra structure} on a $1$-morphism 
$C\in\mathscr{C}(\mathtt{i},\mathtt{i})$
is a pair $(\delta,\varepsilon)$ of morphisms $\delta:C\to CC$ and 
$\varepsilon: C\to \mathbbm{1}_{\mathtt{i}}$ which satisfy the usual coassociativity and counitality axioms
\begin{displaymath}
(\delta\circ_h\mathrm{id})\circ_v\delta=
(\mathrm{id}\circ_h\delta)\circ_v\delta,\quad
\mathrm{id}=(\mathrm{id}\circ_h\varepsilon)\circ_v\delta,\quad
\mathrm{id}=(\varepsilon\circ_h\mathrm{id})\circ_v\delta.
\end{displaymath}

In the case of fiat 2-categories, it is observed in
\cite[Section~6]{MMMT} that a
Duflo 1-morphism often has the structure of a coalgebra 1-morphism 
(as suggested by the existence of a map from the identity to a Duflo $1$-morphism) This is particularly interesting as
it is shown in \cite{MMMT} that any simple transitive 2-representation 
of a fiat 2-category can be constructed using categories of
certain comodules over coalgebra 1-morphisms.

Let $(A,\mu,\eta)$ be an algebra $1$-morphism in $\mathscr{C}$. 
A {\em right module} over $A$ is a pair $(M,\rho)$, where
$M$ is a $1$-morphism in $\mathscr{C}$ and $\rho:MA\to M$
is such that the usual associativity and unitality axioms are 
satisfied:
\begin{displaymath}
\rho\circ_v(\rho\circ_h\mathrm{id})=
\rho\circ_v(\mathrm{id}\circ_h\mu),\quad
\mathrm{id}=\rho\circ_v(\mathrm{id}\circ_h\eta).
\end{displaymath}
Dually, one defines the notion of a comodule over
a coalgebra. Morphisms between (co)modules are defined
in the obvious way. We denote by $\mathrm{mod}_{\mathscr{C}}(A)$
the category of all right $A$-modules in $\mathscr{C}$,
and by $\mathrm{comod}_{\mathscr{C}}(C)$
the category of all right $C$-comodules in $\mathscr{C}$.

\subsection{Coalgebra structure on Duflo $1$-morphisms}\label{s8.2}

Given the results from the previous section, it is 
natural to ask whether $M_k$ is a 
coalgebra 1-morphism in $\mathscr{D}$.

\begin{prop}\label{prop_coalgebra}
For a positive integer $k$, the $1$-morphism $M_k$ 
has the structure of a coalgebra 1-morphism in $\mathscr{D}$. 
Moreover, the $1$-morphism $N_k$ has the structure of 
a right $M_k$-module.
\end{prop}

\begin{proof}
Recall the standard basis of the bimodule
$M_k$ from Section~\ref{sectionhom}.
The tensor product $M_k\otimes M_k$ has a unique direct summand 
isomorphic to $M_k$ with a basis given by 
\begin{displaymath}
\xymatrix@=12pt{
m_2\otimes m_1 & m_2\otimes m_2\ar[l]\ar[d]\\
& m_3\otimes m_2 &  \ar[l]\ddots \ar[d]\\
&& m_{2k+1}\otimes m_{2k} & m_{2k+2}\otimes m_{2k+2}\ar[l]\ar[d] \\
&&& m_{2k+3}\otimes m_{2k+2}.}
\end{displaymath}
Moreover, we have 
$m_{2j+1}\otimes m_{2j}= m_{2j+2}\otimes m_{2j+1}$, for $j=1,\ldots ,k$.
We define the comultiplication $\delta :M_k\to M_k\otimes M_k$ 
explicitly as follows:
\begin{displaymath}
\left\lbrace
\begin{array}{rcll}
\delta (m_{2j}) &=& m_{2j}\otimes m_{2j}  , & 1\leq j\leq k+1\\
\delta (m_{2j+1}) &=& m_{2j+1}\otimes m_{2j}= m_{2j+2}\otimes m_{2j+1}, &1\leq j\leq k+1\\
\delta (m_1)&=& m_2\otimes m_1&\\
\delta (m_{2k+3})&=&m_{2k+3}\otimes m_{2k+2}
\end{array}
\right.
\end{displaymath}
As a counit, we take the morphism $\varphi_k$ from 
Section~\ref{sectionhom}. The counitality and
comultiplication axioms are now checked by a lengthy
but straightforward computation.

To prove that $N_k$ is a right $M_k$-comodule, 
we recall that $N_k\simeq M_k/\mathrm{span}\{ m_{2k+3}\}$. 
Let $\pi :M_k\to N_k$ be the canonical projection. 
Then $\rho = \pi\circ_h \mu$ makes $N_k$ a right $M_k$-comodule. 

Indeed, all necessary properties for $\rho$ follow directly from
the corresponding properties for $\mu$.
\end{proof}

\begin{cor}\label{cormkrep}
The $2$-representation 
$\mathscr{C}M_k\subset \mathrm{comod}_{\mathscr{C}}(C)$ 
of $\mathscr{C}$
has a unique simple transitive quotient, moreover, this quotient
is equivalent to the cell $2$-representation
$\mathbf{C}_{\mathcal{L}}$, where $\mathcal{L}=\{M_k,N_k\}$.
\end{cor}

\begin{proof}
As $M_k$ is indecomposable, the unique simple transitive
quotient $\mathbf{M}$ of $\mathscr{C}M_k$ is the quotient of 
$\mathscr{C}M_k$ by the sum of all $\mathscr{C}$-stable ideals 
in $\mathscr{C}M_k$ which do not contain $\mathrm{id}_{M_k}$. 
Clearly, $M_k$ does not annihilate $M_k$. At the same time,
for any $F>_J M_k$, we have that $\mathscr{C}FM_k$
does not contain $\mathrm{id}_{M_k}$. Therefore any such 
$F$ is killed by $\mathbf{M}$. This means that 
$\mathbf{M}$ has apex $\mathcal{J}_k$. 

Further, $N_kM_k$ does not have any copy of $M_k$
as a direct summand. Therefore the rank of $\mathbf{M}$
is at least $2$. Now the claim of our corollary follows
from Theorem~\ref{mainthm}(iii).
\end{proof}

\subsection{Algebra structure on co-Duflo algebra 1-morphisms}

Similarly to the previous section, it is natural to ask whether 
$W_k$ is an algebra 1-morphism in $\mathscr{D}$. 

\begin{prop}\label{prop_algebra}
For a positive integer $k$, the $1$-morphism $W_k$ 
has the structure of an algebra 1-morphism in $\mathscr{D}$. 
Moreover, the $1$-morphism $S_k$ has the structure of 
a right $W_k$-module.
\end{prop}

\begin{proof}
The tensor product $W_k\otimes W_k$ has a unique direct 
summand isomorphic to $W_k$, namely, the direct summand 
with the basis
\begin{displaymath}
\xymatrix@=12pt{
w_1\otimes w_1\ar[d] \\
w_2\otimes w_1 & w_3\otimes w_3\ar[l]\ar[d]\\
&\ddots &w_{2k-1}\otimes w_{2k-1} \ar[l] \ar[d]\\
&&w_{2k}\otimes w_{2k-1} & w_{2k+1}\otimes w_{2k+1}\ar[l],
}
\end{displaymath}
moreover, $w_{2j}\otimes w_{2j-1}=w_{2j+1}\otimes w_{2j}$, 
for $1\leq j\leq k$. This allows us to define multiplication $\mu$ 
as the projection onto this direct summand. As the unit 
morphism, we take $\psi_k$ from Section~\ref{sectionhom}.
All necessary axioms are checked by a straightforward
computation.

The projection onto the unique summand of $S_k\otimes W_k$ isomorphic to $S_k$ provides $S_k$ with the structure of a right $W_k$-module.
Note that letting $\theta :M_k\to S_k$ and $\zeta: M_k\to W_k$ be the canonical projections (see Section~\ref{s_hom_string}),
and $\pi_{M_k}:M_k\otimes M_k\to M_k$ the projection as in the proof of Proposition~\ref{propduflo-1},
the projection $S_k\otimes W_k\to W_k$ makes the following diagram commute.
\begin{displaymath}
\xymatrix{
M_k\otimes M_k\ar[r]^{\pi_{M_k}}\ar[d]_{\theta\otimes \zeta} & M_k\ar[d]^\theta \\
S_k\otimes W_k\ar[r] & S_k
}
\end{displaymath}
Verifying that this gives $S_k$ the structure of a right $W_k$-module is done by straightforward computation.
\end{proof}

\begin{cor}\label{cormkrep-2}
The $2$-representation 
$\mathscr{C}W_k\subset \mathrm{comod}_{\mathscr{C}}(C)$ 
of $\mathscr{C}$
has a unique simple transitive quotient, moreover, this quotient
is equivalent to the cell $2$-representation
$\mathbf{C}_{\mathcal{L}}$, where $\mathcal{L}=\{W_k,S_k\}$.
\end{cor}

\begin{proof}
Mutatis mutandis Corollary~\ref{cormkrep}.
\end{proof}

\subsection{Rank $1$ representations are non-constructible}
\label{s-nonc}

In this last subsection we would like to emphasize one
major difference between the $2$-representation
theory of $\mathscr{D}$ and that of fiat $2$-categories.

\begin{definition}\label{def_constructible}
Let $\mathscr{C}$ be a (finitary) 2-category and let
$\mathscr{B}\in \{\mathscr{C},\underline{\mathscr{C}},\overline{\mathscr{C}} \}$. 
A 2-rep\-re\-sen\-tation $\mathbf{M}$ of 
$\mathscr{C}$ is called \emph{$\mathscr{B}$-constructible} 
if there is a (co)algebra 1-morphism $C$ in 
$\mathscr{B}$, a $\mathscr{C}$-stable subcategory 
$\mathcal{X}$ of the category of right $C$-(co)modules, 
and a $\mathscr{C}$-stable ideal $\mathcal{I}$ in $\mathcal{X}$ 
such that $\mathbf{M}$ is equivalent to $\mathcal{X}/\mathcal{I}$.
\end{definition}

If $\mathscr{C}$ is fiat, then any simple transitive
$2$-representation of $\mathscr{C}$ is both
$\underline{\mathscr{C}}$- and 
$\overline{\mathscr{C}}$-constructible by \cite{MMMT}.
From \cite[Section~3]{MMMTZ} it follows that 
faithful simple transitive $2$-representation
of $\mathcal{J}$-simple fiat $2$-categories are
even $\mathscr{C}$-constructible.

Corollary~\ref{cormkrep} implies that, for each $k\geq 1$, the 
cell  2-representation $\mathbf{C}_{\mathcal{L}}$
of $\mathscr{D}^{(k)}$, where $\mathcal{L}=\{ M_k,N_k\}$, 
is $\mathscr{D}^{(k)}$-constructible.

The following statement, in some sense, explains why 
the statement of Theorem~\ref{mainthm}(iv) is
as it is.

\begin{theorem}\label{nonconstr}
Let $k$ and $m$ be positive integers such that $2\leq k\leq m$.
Let $\mathbf{M}$ be a rank $1$ simple transitive 
2-representation of $\mathscr{D}^{(m)}$ with apex $\mathcal{J}_k$. 
Then $\mathbf{M}$ is not $\mathscr{D}^{(m)}$-constructible.
\end{theorem}

\begin{proof}
Assume towards contradiction that $\mathbf{M}$ is $\mathscr{D}$-constructible. Let $\mathcal{X}$ be as in Definition~\ref{def_constructible} and consider some object $X\in \mathcal{X}$ which is nonzero in the quotient by $\mathcal{I}$. Then, for each $U\in \mathcal{J}_k$,  we must have $UX\simeq X + \mathcal{I}$.

If $X\in \mathrm{add} \{ \mathcal{J}\,\vert\,\mathcal{J}>_J \mathcal{J}_k\}$, then the action of $\mathcal{J}_{k-1}$  
is nonzero on $X$, implying that $\mathcal{J}_k$ is not the apex of the representation (note that $k>1$). 
This means that all indecomposable summands of $X$ which matter
for the computations in $\mathcal{X}/\mathcal{I}$ are in 
$\mathcal{J}_{k}$.

From \eqref{eqtable}, we obtain that, modulo
higher two-sided cells, $N_kX_m\in \mathrm{add}\{ N_k\oplus W_k\}$
while $M_kX_m\in \mathrm{add}\{ M_k\oplus S_k\}$.
Since $\mathcal{J}_{k}$ is the apex of $\mathbf{M}$,
both $N_kX_m$ and $M_kX_m$ are non-zero.
This contradicts the assumption that $\mathbf{M}$ has rank $1$.
\end{proof}

\vspace{5mm}

\noindent
Department of Mathematics, Uppsala University, Box. 480,
SE-75106, Uppsala, SWEDEN, email: {\tt helena.jonsson\symbol{64}math.uu.se}

\end{document}